\documentclass[12pt, a4paper, oneside]{article}
\usepackage{amsmath, amsthm, amssymb, bm, graphicx,  mathrsfs,geometry,color,cite}
\usepackage[title]{appendix} 
\usepackage{lineno} 
\usepackage{booktabs} 
\usepackage{array} 
\usepackage{caption}
\usepackage{titling} 
\usepackage{authblk}
\usepackage{etoolbox} 
\usepackage{subcaption}
\usepackage[hyperfootnotes=false]{hyperref}

\newtheorem{theorem}{Theorem}[section]
\numberwithin{equation}{section}

\newtheorem{lemma}[theorem]{Lemma}
\newtheorem{corollary}[theorem]{Corollary}

\newtheorem{proposition}[theorem]{Proposition}
\newtheorem{remark}[theorem]{Remark}

\renewcommand{\figurename}{Fig.}

\usepackage{titlesec}
\titleformat{\section}
{\fontsize{14}{16}\selectfont\bfseries}
{\thesection}{1em}{}

\pretitle{\begin{center}\fontsize{20}{24}\selectfont\bfseries}      
	\posttitle{\end{center}\vskip 1em}           
\preauthor{\begin{center}\large}             
	\postauthor{\end{center}}                    
\predate{\begin{center}\small}               
	\postdate{\end{center}}                      
\date{}     

\title{\textbf{Principal spectral theory and variational characterizations for nonlocal coupled cooperative systems and applications}}

\author[a]{Xiandong Lin}
\author[b]{Jiazhuo Cheng}
\author[a]{Qiru Wang \thanks{Corresponding Author: Email: mcswqr@mail.sysu.edu.cn}}

\affil[a]{School of Mathematics, Sun Yat-sen University, Guangzhou 510275, Guangdong, P.R. China}
\affil[b]{School of Mathematics and Systems Science, Guangdong Polytechnic Normal University, Guangzhou 510665, Guangdong, P.R. China}

\date{}

\begin{document}
	
	\maketitle
	
	\vspace{-3em}
	
	\begin{abstract}
		This paper investigates the principal spectral theory of a nonlocal dispersal operator with coupled diffusion and aims to establish a variational characterization of the spectral bound for the case where the system is not strongly coupled. In this setting, a key difficulty arises since the principal eigenfunction may have components that are identically zero, rendering existing generalized eigenvalue methods inapplicable. To overcome this, we reorder the components of the operator using a permutation matrix, thereby decomposing it into suitable suboperators, and characterize the spectral bound of the original operator in terms of the spectral bounds of these suboperators. Building on this principal spectral theory, we provide a variational characterization of the basic reproduction ratio for nonlocal dispersal systems and analyze the dynamical behavior of a class of multi‑genotype stem cell regeneration models with epigenetic transitions, both in the presence and absence of gene mutations, without assuming the existence of a principal eigenvalue. Furthermore, we investigate the threshold dynamics when the system is not strongly coupled.
	\end{abstract}
		


\section{Introduction}

This paper considers the following nonlocal dispersal operator on $C(\overline\Omega, \mathbb{R}^{n}) $:
\[
\left[ L\phi\right](x) = D(x)\left[ \mathcal{J}\phi\right] (x) + A(x)\phi(x), \quad x\in \overline\Omega,
\]
where
$\Omega\subset \mathbb{R}^{m}$ is a bounded domain with smooth boundary,
$D(x):=(d_{ij}(x))_{n\times n}$ is a nonnegative matrix and $A(x):=(a_{ij}(x))_{n \times n}$ is a cooperative matrix with $d_{ij}, a_{ij} \in C(\overline\Omega, \mathbb{R})$ and $d_{ii}(x)>0$ for each $i, \ j\in \mathcal{S}:=\left\lbrace 1, 2, \dots, n \right\rbrace$, and 
$ \mathcal{J}$ is a operator on $C(\overline\Omega, \mathbb{R}^{n}) $ defined as 
\[
\left[ \mathcal{J}\phi\right] (x):= \left( \int_{\Omega}J_{1}(x,y)\phi_{1}(y)\, \mathrm{d}y, \int_{\Omega}J_{2}(x,y)\phi_{2}(y)\, \mathrm{d}y, \dots, \int_{\Omega}J_{n}(x,y)\phi_{n}(y)\, \mathrm{d}y \right)^{T},  \quad x\in \overline\Omega.
\]
Here 
$J_{i}(x,y)$ is a continuous function of $(x,y)\in\mathbb{R}^{2m} $ with 
\[
J_{i}(x,x)>0, \quad \int_{\mathbb{R}^{m}}J_{i}(y,x)\, \mathrm{d}y=1, \quad  \text{ for all } x\in \mathbb{R}^{m}, i \in \mathcal{S}.
\]

The positive coupling term $D\mathcal{J}$ in $L$ originates from stem cell regeneration models \cite{MR4070771,MR4278592,MR4612706} and can be interpreted as nonlocal cross-diffusion. Spectral properties, especially the principal eigenvalue, play a critical role in the dynamics of such models \cite{MR4612706}. 

In the scalar case $n=1$, $L$ has been widely used to model long-range diffusion and transitions between multiple states. Its principal eigenvalue has been extensively studied in \cite{MR3498523,MR2257732,MR2718672,MR3639154,MR3632209,MR2028048}, for further works on time-periodic case see \cite{MR2436718,MR3000610,MR4267553,MR4104463,MR4620151,MR4781053}. 

For $n>1$,  Lin and Wang \cite{MR4628895} considered the case where $A$ is weakly irreducible and $D$ is a constant diagonal matrix, with  kernels $J_i(x,y)=k_{i}(x-y)$, and obtained  the existence  criteria of the principal eigenvalue via the generalized Krein–Rutman theorem. Further, Zhang \cite{MR4803717} considered the case where $A$ is pointwise irreducible   and studied the partially degenerate case, that is,  $D(x)=\operatorname{diag}\{1,\dots,1,0,\dots,0\}$, allowing nonlocal dispersal to be absent in some components.   Moreover, in the  the time- and space-periodic setting, Bao and Shen \cite{MR3637938} took $D(x)=I$, pointwise irreducible $A(x,t)$, and  kernels $J_i(x,y)=k(x-y)$, and used perturbation theory for positive semigroups to obtain criteria for the principal eigenvalue. Liang et al. \cite{MR3705788} extended these results to the partially degenerate case using the Krein–Rutman theorem. Feng et al. \cite{MR4754238} further investigated the effect of the diffusion rate on the spectral bound in the time-periodic case.

For the case where  $D$ is a positive matrix, the principal spectral theory of $L$ has been studied in \cite{MR4601060,MR4929554}. Su et al. \cite{MR4601060} investigated   the existence, uniqueness, multiplicity, variational characterizations, spectral bound, essential spectrum, the sign relationship of the principal eigenvalue.   Their analysis was carried out under the assumption that  $L$ is either ``strong" in cooperation or ``strong" in the coupling of the nonlocal dispersal operators(i.e., either $D$  or $A$ is weakly irreducible). Subsequently, Wu et al. \cite{MR4929554} extended the principal spectral theory to unbounded domains and, along with Wang and Zhang \cite{MR5005572}, developed a general method for constructing auxiliary operators with principal eigenvalues to study the spectral bound via asymptotic operators.

In this paper, we focus on the variational characterization of the spectral bound of $L$ without assuming that either $D$ or $A$ is weakly irreducible. 
In the special case where $D$ is a constant diagonal matrix and $A$ is weakly irreducible, Lin et al. \cite{lin2026} have already established such variational characterizations. 
In that setting, the weak irreducibility of $A$ guarantees that when $L$ admits a principal eigenvalue, the corresponding principal eigenfunction is strongly positive. This property allows the application of results from \cite{MR3583503,MR3548277,MR4601060,MR5005572} to derive the Collatz--Wielandt characterization of the spectral bound.

However, when we drop the irreducibility assumption on $D$ or $A$,  the principal eigenfunction of $L$ (if it exists) is no longer guaranteed to be strongly positive. Consequently, the definition of generalized principal eigenvalues of $L$ (see, e.g., \cite{MR4612706,MR3583503,MR2718672,MR5005572}) becomes inapplicable, because in this case generalized eigenvalues may fail to coincide with the spectral bound even when a principal eigenvalue exists.

To overcome this difficulty, we rearrange the components of $D$ and $A$ via a suitable permutation to bring them into block lower triangular form, so that the nonlocal dispersal operators corresponding to the diagonal blocks are eventually essentially strongly positive.
This ensures that if a principal eigenvalue exists, its associated principal eigenfunction is strongly positive.  Building on the results of \cite{lin2026}, we can then characterize the spectral bound using the notion of generalized eigenvalues. Moreover, we prove that the spectral bound of the original operator equals the maximum of the spectral bounds of these block operators, thereby obtaining the variational characterizations of the spectral bound (see Section 3 for details).

As  applications,  we first study a class of cooperative systems with nonlocal and coupled diffusion. 
Based on the characterizations of the spectral bound, we establish the variational characterizations of the basic reproduction ratio, which further extends the results in \cite{lin2026}.

Furthermore, we apply the developed spectral theory to investigate the dynamical behavior of a class of multiple‑genotype stem cell regeneration models with epigenetic transitions, both with and without gene mutations. This problem was previously examined by Su et al.~\cite{MR4612706}. They studied the associated eigenvalue problem and, under the assumption that  a principal eigenvalue exists, used principal eigenvalue theory to analyze the system's dynamics. 
Here, we remove their key assumption that guarantees the existence of a principal eigenvalue. This removal better reflects the essential difference between the study of nonlocal dispersal problems and local diffusion problems.  By applying the more general theory of principal eigenvalues developed in Section 3, we still obtain the same spectral conclusions. Therefore, following the same arguments as in \cite{MR4612706}, the dynamical behavior of the system can be established. Moreover, we extend the analysis to the reducible case, which was not considered in \cite{MR4612706}, and study its dynamics. We further present a numerical method for computing the spectral bound, whose sign determines the dynamical behavior of the system.

The remainder of this paper is organized as follows. In Section 2, we recall  some spectral theory for nonlocal dispersal operator which serve as the foundation for our later study of the spectral bound. In Section 3, we establish  the variational characterizations of the spectral bound. In Section 4, we apply these  characterizations to derive  the variational characterizations of the basic reproduction ratio  for nonlocal dispersal systems. In Section 5, we apply the established spectral theory to investigate the dynamical behavior of a class of multiple‑genotype stem cell regeneration models with epigenetic transitions, both with and without gene mutations.

\section{Preliminaries}

In this section, we recall the existence criteria for the principal eigenvalue of the nonlocal dispersal operator and, for the eventually essentially strongly positive case, the variational characterization of its spectral bound. These results will serve as the foundation for our later study of the spectral bound  in the case where it is not eventually essentially strongly positive.

We begin by introducing some notations.
Denote $X:=C(\overline\Omega, \mathbb{R}^{n})$ equipped with the maximum norm and the positive cone $X^{+}:=C(\overline\Omega, \mathbb{R}^{n}_{+})$. We use $\ge$, $>$, and $\gg$ to denote  the order, strict order, and strong order relations induced by the cone $X^{+}$, respectively; that is,
\[
a\ge b \; \text{ if } \; a-b\in X^{+},\qquad 
a> b \; \text{ if } \; a-b\in X^{+}\setminus\{0\},\qquad 
a\gg b \; \text{ if } \; a-b\in \operatorname{Int}(X^{+}).
\]
In particular, when $b=0$, we say that $a$ is  positive, strictly positive, and strongly positive, respectively.

Let $T$ be a bounded linear operator on $ X $. We say that $T$ is  positive if $ TX^{+} \subseteq X^{+} $, strongly positive if $ T( X^{+} \backslash \left\{ 0 \right\} ) \subseteq {\rm{Int}}\left( X^{+} \right) $, and eventually essentially strongly positive  if for every  $ x \in X^{+} $, there exists $  l = l(x) \in \mathbb{Z}^{+} $ such that $ T^{l}x \in \rm{Int}\left( X^{+} \right) \cup \left\lbrace 0\right\rbrace  $.

Denote by $\sigma(T)$  the spectrum of $T$, and by $s(T)$  its spectral bound, i.e.,
\[
s(T) := \sup \{\operatorname{Re} \mu : \mu \in \sigma(T)\}.
\]
Let $\mathcal{R}(T)$ and $\mathcal{N}(T)$ be the range and null space of $T$, respectively.  $T$ is called a Fredholm operator if $\mathcal{R}(T)$ is closed and both 
\[
\dim \mathcal{N}(T) < \infty \quad \text{and} \quad \operatorname{codim} \mathcal{R}(T) := \dim \big( X / \mathcal{R}(T) \big) < \infty.
\]
Its Fredholm index is then given by
\[
\operatorname{ind}(T) = \dim \mathcal{N}(T) - \operatorname{codim} \mathcal{R}(T).
\]
Following \cite[Section 7.5]{MR1861991}, we define the  essential spectrum of $T$ as
\[
\sigma_e(T) := \bigl\{ \lambda \in \sigma(T) : \lambda I - T \text{ is not a Fredholm operator of index zero} \bigr\}.
\]
The spectral radius and the essential spectral radius of $T$ are denoted by $r(T)$ and $r_e(T)$, respectively. Since Nussbaum \cite{MR264434} showed that various definitions of the essential spectral radius are equivalent, we may adopt the following one:
\[ 
{r_e}\left( T\right):= \sup \left\{ {\left| \lambda  \right|:\lambda  \in {\sigma _e}\left( T\right)} \right\}.
\]
We recall a version of the Krein-Rutman theorem for eventually essentially strongly positive operators, which was proved in \cite{MR4628895}.

\begin{lemma}[{\cite[Lemma 2.2]{MR4628895}}]\label{le2.1}
	Let $ (E,E_{+}) $ be an ordered Banach space with ${\rm{Int}}(E_{+}) \ne 0 $, and  $T$ be an eventually essentially strongly positive operator on $E$. If ${r_e}\left( T \right) < r\left( T \right)$, the following statements are valid:
	
	\begin{itemize}
		\item [\rm(i)] There exists $y \in {\rm{Int}}\left( E_{+} \right)$ such that $ Ty= r\left( T \right) y$;
		
		\item [\rm(ii)] $r\left( T \right)$ is an algebraically simple eigenvalue of $ T  $, i.e.
		\[ \dim \mathop  \cup \limits_{i = 1}^\infty  \ker \left( {{{\left( {r\left( T\right)I - T} \right)}^i}} \right) = 1 ; \]
		
		\item [\rm(iii)] If $  v > 0 $ is an eigenvector associated to a nonzero eigenvalue for $ T $, then there exists $ s > 0 $ such that $ v = sy $;
		
		\item [\rm(iv)] $\left| \lambda  \right| < r\left( T \right)$ for any other $ \lambda \in \sigma \left( T \right) $.
	\end{itemize}
\end{lemma}

The matrices $D(x)$ and $A(x)$ are said to be weakly irreducible in the sense that the index set $\mathcal{S}$ cannot be  split into two disjoint nonempty subsets $\mathcal{I}$ and $\mathcal{J}$ such that,  $d_{ij}(x) \equiv 0$ (resp. $a_{ij}(x) \equiv 0$) for all $i \in \mathcal{I}$ and $j \in \mathcal{J}$  on $\overline{\Omega}$.

\begin{proposition}
	If $D(x)$ or $A(x)$  is  weakly irreducible, then there exists  $c_{0}$ such that $L+c_{0}I$ is an eventually essentially strongly positive operator.
\end{proposition}

\begin{proof}
	Clearly, there exists a constant $c_0>0$ such that $L+c_0I$ is positive. Moreover, by an argument analogous to that in the proof of \cite[Proposition 2.1]{MR4628895}, one can show that $L + c_0 I$ is eventually essentially strongly positive.
\end{proof}

We remark that  $D(x)$ or $A(x)$  is  weakly irreducible if and only if $D(x)+A(x)$  is  weakly irreducible. And if $A(x)$ is a nonnegative matrix, then we can simply take $c_0 = 0$.

Since $D\mathcal{J}$ is a compact operator on $X$, according to  \cite[Theorem 7.26]{MR1861991} and \cite[Proposition 2.7]{MR3906242}, we have $\sigma_{e}(L)=\bigcup_{x \in \overline\Omega}  \sigma(A(x))$. Note that we can choose a $c_{0}$ such that $A(x)+c_{0}I$ is a nonnegative matrix and $L+c_{0}I$ is positive. According to \cite[p.276]{MR1921782}, we have $s(L+c_{0}I) =r(L+c_{0}I)$ and $s(A(x)+c_{0}I)=r(A(x)+c_{0}I)$. Clearly, $r(L+c_{0}I)= s(L)+c_{0} $ and $r(A(x)+c_{0}I)= \max\limits_{x\in \overline{\Omega}}s(A(x))+c_{0}$. 

$L$ is said to admit a principal eigenvalue if $s(L)$ is an eigenvalue with a positive eigenfunction. According to Krein-Rutman theorem (see \cite[Corollary 2.2]{MR643014}) and Lemma \ref{le2.1}, we have the following results.

\begin{proposition}\label{pr2.1}
	$L$  admits the principal eigenvalue if $s(L)> \max\limits_{x \in \overline{\Omega}}s(A(x)) $.
\end{proposition}

\begin{theorem}\label{th2.1}
Assume that there exists  $c_{0}$ such that $L+c_{0}I$ is an eventually essentially strongly positive operator. If $s(L)> \max\limits_{x\in \overline{\Omega}}s(A(x))$, then 
\begin{itemize}
	\item [\rm(i)] There exists $y \in {\rm{Int}}\left( X^{+} \right)$ such that $ Ly= s\left(L \right) y$;
	
	\item [\rm(ii)] $s\left( L\right)$ is an algebraically simple eigenvalue of $ L  $, i.e.
	\[ \dim \mathop  \cup \limits_{i = 1}^\infty  \ker \left( {{{\left( {s\left( L\right)I - L} \right)}^i}} \right) = 1 ; \]
	
	\item [\rm(iii)] If $  v > 0 $ is an eigenvector associated to a nonzero eigenvalue for $ L$, then there exists $ l > 0 $ such that $ v = ly $;
	
	\item [\rm(iv)] $\left| \lambda  \right| < s\left( L \right)$ for any other $ \lambda \in \sigma \left( L\right) $.
\end{itemize}
\end{theorem}

By a similarly arguments as in \cite[Theorem 2.3]{MR4628895} and \cite[Lemma 2.5]{lin2026}, we have the following results.

\begin{proposition}
	The following statements hold:
	\begin{itemize}
		\item [\rm(i)]  Assume that there exists  $c_{0}$ such that $L+c_{0}I$ is an eventually essentially strongly positive operator. Then $ s(L) $ is the principal eigenvalue of $ L $ if and only if $ s(L)> \mathop  {\max} \limits_{x\in\overline{\Omega}} s(A(x)) $;
		
		\item [\rm(ii)] 	If there exists an open set $\Omega_{0}\subset \Omega$ such that $s(A(x))=  \max\limits_{x \in \overline{\Omega}}s(A(x) )$ for all $x\in \Omega_{0}$, then $s(L)> \max\limits_{x \in \overline{\Omega}}s(A(x) )$. 
	\end{itemize}
\end{proposition}

\begin{lemma}[{\cite[Theorem 2.6]{lin2026}}]\label{le2.2}
Assume that there exists  $c_{0}$ such that $L+c_{0}I$ is an eventually essentially strongly positive operator. Then
\[
\begin{aligned}
	s(L) =& \inf \left\lbrace \lambda\in \mathbb{R}: \exists \phi \in \mathrm{Int}(X^{+}) \ s.t. \ \left[ L\phi\right](x)\le \lambda \phi(x), \forall x\in \overline{\Omega}  \right\rbrace\\[0.5em]
	=& \inf_{\phi \in \mathrm{Int}(X^{+})} \sup_{x \in \Omega, i \in \mathcal{S}}  \frac{\sum_{j=1}^n d_{ij}(x) \int_{\Omega} J_j(x- y) \phi_j(y) \,\mathrm{d}y + \sum_{j=1}^n a_{ij}(x)\phi_{j}(x) }{\phi_i(x)}\\[0.5em]
	=&\sup \left\lbrace \lambda\in \mathbb{R}: \exists \phi \in \mathrm{Int}(X^{+}) \ s.t. \ \left[ L\phi\right](x)\ge \lambda \phi(x), \forall x\in \overline{\Omega}   \right\rbrace\\[0.5em]
	=& \sup_{\phi \in  \mathrm{Int}(X^{+})} \inf_{x \in \Omega, i \in \mathcal{S}}  \frac{\sum_{j=1}^n d_{ij}(x) \int_{\Omega} J_j(x- y) \phi_j(y) \,\mathrm{d}y + \sum_{j=1}^n a_{ij}(x)\phi_{j}(x) }{\phi_i(x)}.
\end{aligned}
\]		
\end{lemma}

\section{Variational characterizations of the spectrum bound}

In this section, we focus on establishing  a variational characterization of the spectral bound for the case where $L+cI$ is not eventually essentially strongly positive for any $c\ge 0$.

We first prove a property of $L$ that plays an important role in the characterization of the spectral bound.

\begin{proposition}\label{pr3.1}
Let $c_{0}$ be a constant for which $L+c_{0}I$ is positive. Then there exists a  permutation matrix $P$ such that 
{\small	\[
	PD(x)P^{T}=\begin{pmatrix} D_{11}(x) & 0 & \cdots & 0 \\ D_{21}(x) & D_{22}(x) & \cdots & 0 \\ \vdots & \vdots & \ddots & \vdots \\ D_{\tilde{n}1}(x) & D_{\tilde{n}2}(x) & \cdots & D_{\tilde{n}\tilde{n}}(x) \end{pmatrix}, \quad  PA(x)P^{T}=\begin{pmatrix} A_{11}(x) & 0 & \cdots & 0 \\ A_{21}(x) & A_{22}(x) & \cdots & 0 \\ \vdots & \vdots & \ddots & \vdots \\ A_{\tilde{n}1}(x) & A_{\tilde{n}2}(x) & \cdots & A_{\tilde{n}\tilde{n}}(x) \end{pmatrix},
	\]}
	and $L_{k} + c_{0} I$ is an eventually essentially strongly positive operator on $X_{k}$ for all $1 \le k \le \tilde n$, where
	\[
	\left[ L_{k}\psi\right](x) = D_{kk}(x)\left[ \mathcal{J}_{k}\psi\right] (x) + A_{kk}(x)\psi(x), \quad \psi \in X_{k}.
	\]
	Here $X = X_1 \times \cdots \times X_{\tilde n}$ with $X_k = C(\overline\Omega,\mathbb{R}^{m_k})$, where $m_k$ is the size of the $k$-th diagonal block, and $\mathcal{J}\phi = (\mathcal{J}_1\phi_1,\dots,\mathcal{J}_{\tilde n}\phi_{\tilde n})^T$ for $\phi = (\phi_1,\dots,\phi_{\tilde n})\in X$.
\end{proposition}

\begin{proof}

If  $L+c_{0} I$ is an eventually essentially strongly positive operator, the conclusion follows trivially by taking $P=I$. Otherwise, assume that $L+c_{0} I$ is not eventually essentially strongly positive and prove the result by mathematical induction.

Since $J_i(x,x)>0$, it is easy to see that for any $\phi \in X^+$ there exists $l:=l(\phi)$ such that $u:=L^l\phi$ satisfies either $u_i(x)>0$ for all $x\in\overline\Omega$ or $u_i\equiv 0$. Consequently, $L$ is eventually essentially strongly positive when $n=1$, so the conclusion holds for $n=1$.

Now assume that the conclusion holds for all $n\le r$. We  prove it for $n=r+1$. Since $L+c_{0} I$ is not eventually essentially strongly positive, there exist $\phi\in X^+$ and $l=l(\phi)$ such that $u:=\left( L+c_{0} I\right) ^l\phi$ satisfies $u_i(x)>0$ for some $i\in\mathcal{S}$ and $u_j\equiv0$ for some $j\in\mathcal{S}$. Then $\mathcal{S}$ can be split into two disjoint nonempty sets
\[
\mathcal{I}:=\{i\in\mathcal{S}: u_i(x)>0,\ \forall x\in\overline\Omega\},\qquad 
\mathcal{J}:=\{j\in\mathcal{S}: u_j(x)=0,\ \forall x\in\overline\Omega\}.
\]
Without loss of generality, after a suitable permutation $P_1$, we may assume that $\mathcal{J}=\{1,\dots,s\}$ and $\mathcal{I}=\{s+1,\dots,n\}$ with $0<s<n$.

Set $v:=\left( L+c_{0}\right)^{l-1} \phi$. For each $j\in\mathcal{J}$, we have
\[
0=\sum_{i\in\mathcal{I}} d_{ji}(x)\int_\Omega J_i(x,y)v_i(y)\,dy+\sum_{i\in\mathcal{I}} a_{ji}(x)v_i(x).
\]
It follows that $d_{ji}(x)\equiv0$ and $a_{ji}(x)\equiv0$ for all $i\in\mathcal{I}$ and $j\in\mathcal{J}$. Hence
\[
P_1 D(x) P_1^T=\begin{pmatrix} D_{11}(x) & 0 \\ D_{21}(x) & D_{22}(x) \end{pmatrix},\qquad 
P_1 A(x) P_1^T=\begin{pmatrix} A_{11}(x) & 0 \\ A_{21}(x) & A_{22}(x) \end{pmatrix},
\]
where the blocks correspond to the splitting $\mathcal{J}\cup\mathcal{I}$. Since $s\le r$ and $ n-s \le r$, by the induction hypothesis, the result follows for $n=r+1$.	
\end{proof}

From Proposition \ref{pr3.1}, there exists a permutation matrix such that 
$D$ and  $A$ become lower triangular block matrices. Consequently, without loss of generality, we make the following assumption:
\begin{itemize}
	\item [(\textbf{LT})] $D$ and $A$ possess  the lower triangular block form described in Proposition \ref{pr3.1}.
\end{itemize}

\begin{proposition}\label{pr3.2}
	Assume that {\rm(\textbf{LT})} holds. Then
	$s(L)= \max\limits_{1\le k\le \tilde{n}} s(L_{k})$.
	
\end{proposition}

\begin{proof}
Denote 
\[
\tilde{L}:=\text{diag}\left\lbrace L_{1}, \dots, L_{\tilde n}\right\rbrace. 
\]
It is easy to see that $\sigma(\tilde{L})= \cup_{k =1}^{ \tilde n}\sigma(L_{\tilde k})$. Notice that 
$
L\phi\ge \tilde{L}\phi 
$
for any $\phi\in X^{+}$. By  the Gelfand's formula (see, e.g., \cite[Theorem
VI.6]{MR751959}), we obtain that $r(L+c_{0}I)\ge r(\tilde L+c_{0}I)$,  where  $c_{0}$ is chosen such that   $ \tilde L+c_{0}I$ is positive. 
Recall that the spectral radius of a positive operator belongs to the spectrum if the order Banach space is a Banach lattice (see, e.g., \cite[p.276]{MR1921782}). Moreover, one can verify that $(X,X^{+})$ is a  Banach lattice.  Then 
$s(L) +c_{0}= r(L+c_{0}I)\ge r(\tilde{L}+c_{0}I) = s(\tilde L)+c_{0}$.

On the other hand, according to \cite[Corollary 4]{MR1618686}, we have $\sigma(L)\subset \sigma(\tilde{L})$, which implies  $s( L) \le s(\tilde L)  $. Therefore, $s(L)= s(\tilde L)=\max\limits_{1\le k\le \tilde{n}} s(L_{k})$.
Then we have $s(L)=\max\limits_{1\le k\le \tilde{n}} s(L_{k})$.
\end{proof}

By the same argument, we obtain the following results.

\begin{corollary}
	Assume that {\rm(\textbf{LT})} holds. Then
	$\max\limits_{x\in \overline{\Omega}}s(A(x))= \max\limits_{1\le k\le \tilde{n}} \max\limits_{x\in \overline{\Omega}}s(A_{kk}(x))$.
\end{corollary}

Set $m_{0} = 0$. Denote $\mathcal{S}_{k}:=\left\lbrace m_{k-1}+1, \dots, m_{k} \right\rbrace $ and $X_{k}^{+}= C(\overline\Omega,\mathbb{R}_{+}^{m_k})$ for  $1\le k \le \tilde n$,  with $m_{k}$ and $\tilde n $  as given in Proposition \ref{pr3.1}.  According to Lemma \ref{le2.2} and Proposition \ref{pr3.2}, we can derive the following conclusion.

\begin{theorem}\label{th3.1}
	Assume that {\rm(\textbf{LT})} holds. Then
	\[ \begin{aligned}
	s(L)=& \max\limits_{1\le k\le \tilde{n}} \inf \left\lbrace \lambda\in \mathbb{R}: \exists \phi \in \mathrm{Int}(X_{k}^{+}) \ s.t. \ \left[ L_{k}\phi\right](x)\le \lambda \phi(x), \forall x\in \overline{\Omega}  \right\rbrace\\[0.5em]
	=& \max\limits_{1\le k\le \tilde{n}}\inf_{\phi \in \mathrm{Int}(X_{k}^{+})} \sup_{x \in \Omega, i \in \mathcal{S}_{k}}  \frac{\sum_{j\in \mathcal{S}_{k} } d_{ij}(x) \int_{\Omega} J_j(x,y) \phi_j(y) \,\mathrm{d}y + \sum_{j\in \mathcal{S}_{k}} a_{ij}(x)\phi_{j}(x) }{\phi_i(x)}\\[0.5em]
	=&\max\limits_{1\le k\le \tilde{n}} \sup \left\lbrace \lambda\in \mathbb{R}: \exists \phi \in \mathrm{Int}(X_{k}^{+}) \ s.t. \ \left[ L_{k}\phi\right](x)\ge \lambda \phi(x), \forall x\in \overline{\Omega}   \right\rbrace\\[0.5em]
	=& \max\limits_{1\le k\le \tilde{n}} \sup_{\phi \in  \mathrm{Int}(X_{k}^{+})} \inf_{x \in \Omega, i \in \mathcal{S}_{k}}  \frac{\sum_{j\in \mathcal{S}_{k}} d_{ij}(x) \int_{\Omega} J_j(x, y) \phi_j(y) \,\mathrm{d}y + \sum_{j\in \mathcal{S}_{k}} a_{ij}(x)\phi_{j}(x) }{\phi_i(x)}.
\end{aligned}
	\]
\end{theorem}

\section{Variational characterizations of the basic reproduction ratio}

In this section, we will apply the variational characterization of the spectral bound established in Section 3 to obtain a variational characterization of the basic reproduction ratio. To illustrate this, we first introduce the following multi‑strain SIS epidemic model: 
\begin{equation}\label{eq4.1}
	\begin{cases} 
	\displaystyle	\dfrac{\partial u }{\partial t} = d_{u} \int_{\Omega}K(x-y) \left(  u(y)-u(x) \right) \, \mathrm{d}y -\dfrac{u \sum_{i=1}^k p_i(x) v_i}{u + \sum_{j=1}^k v_j} + \sum_{i=1}^k q_i(x) v_i, & (t,x) \in \mathbb{R}^+ \times \Omega, \\[1em]
		\displaystyle	\dfrac{\partial v_i}{\partial t}  = d_{v_{i}} \int_{\Omega}K(x-y) \left(  v_{i}(y)-v_{i}(x) \right) \, \mathrm{d}y + p_i(x) \dfrac{u v_i}{u + \sum_{j=1}^k v_j} - q_i(x) v_i, & (t,x) \in \mathbb{R}^+ \times \Omega, \\[1em]
		\displaystyle\int_{\Omega} \bigl( u(x,0) + \sum_{i=1}^k v_i(0,x) \bigr)\, \mathrm{d}x =N> 0,
	\end{cases}
\end{equation}
where $k\ge 1$, $K(x)$ is a nonnegative continuous dispersal kernel function satisfying $K(0)>0$ and $K(-x)=K(x)$. Here, $v_i(x,t)$ denotes the density of individuals infected with the $i$-th strain at time $t$ and location $x\in\Omega$ ($i=1,\ldots,k$), where $d_{v_i}>0$ are the diffusion coefficients of the infected subpopulations, and $p_i(x)$ and $q_i(x)$ are the infection and recovery rates for the $i$-th strain, respectively. System \eqref{eq4.1} extends the single‑strain SIS model studied in \cite{MR4620151,MR3945624}. When the diffusion process is modeled using the Laplace operator, system \eqref{eq4.1} has been studied in \cite{MR2897882,MR4344573,MR4609520,MR4725138,MR2002704} and references therein. Moreover, when the diffusion process incorporates L\'evy flights (i.e., fractional diffusion), system \eqref{eq4.1} was studied by Shi et al. \cite{MR5013092}.

Clearly, the disease-free equilibrium of the system \eqref{eq4.1} is $(\frac{N}{\left| \Omega\right| }, 0, \dots, 0)^{T}$. Linearizing the system at the disease-free equilibrium, we obtain:
\begin{equation}\label{eq4.2}
	\frac{\partial v_i}{\partial t} = d_{v_i} \int_{\Omega} K(x-y) \bigl( v_i(y,t) - v_i(x,t) \bigr) \, \mathrm{d}y + \bigl( p_i(x) - q_i(x) \bigr) v_i(x,t), \qquad i=1,\dots,k.
\end{equation}
Applying the theory of the basic reproduction ratio (see \cite{MR3032845,MR3612966,MR3992071,MR2505085}), we can define the basic reproduction ratio for system \eqref{eq4.1} based on the linearized system \eqref{eq4.2}.

In order to make the characterization of the basic reproduction ratio more general, we rewrite system \eqref{eq4.2} in a more general form: 
\begin{equation}\label{eq5.1}
	\begin{cases}
		\dfrac{	\partial u}{\partial t}= D(x)[\mathcal{J}u(t,\cdot)](x) + V(x)u(t,x) +  F(x)u(t,x), & t>0,  x \in \Omega,\\
		u(0,x)=u_{0}(x), & x\in \Omega,
	\end{cases}
\end{equation}
and consider the following nonlocal dispersal systems:
\begin{equation}\label{eq5.2}
	\begin{cases}
		\dfrac{	\partial u}{\partial t}= D(x)[\mathcal{J}u(t,\cdot)](x) + V(x)u(t,x), & t>0,  x \in \Omega,\\
		u(0,x)=u_{0}(x), & x\in \Omega,
	\end{cases}
\end{equation}
where $ V(x) = \bigl(v_{ij}(x)\bigr)_{m \times m} $ is a cooperative matrix and $ F(x) = \bigl(f_{ij}(x)\bigr)_{m \times m} $ is a  nonnegative matrix, with $v_{ij}, f_{ij}\in C(\overline\Omega, \mathbb{R})$. Note that $s(F(x))\ge 0$ for all  $x\in\overline\Omega$. If $\max\limits_{x\in\overline\Omega}s(F(x))=0 $,  we may replace $F$, $V$ by $\delta I + F $,  $V-\delta I$ for any $\delta>0$, respectively.  Hence, without loss of generality, we assume throughout this section that $\max\limits_{x\in\overline\Omega}s(F(x))>0$ holds.
By the same argument as in Proposition \ref{pr3.1}, after a suitable permutation if necessary, the matrices $D$, $V$ and $F$ admit a lower triangular block form. Without loss of generality, we assume
\[
D(x)=\begin{pmatrix} D_{11}(x) & 0 & \cdots & 0 \\ D_{21}(x) & D_{22}(x) & \cdots & 0 \\ \vdots & \vdots & \ddots & \vdots \\ D_{\tilde{n}1}(x) & D_{\tilde{n}2}(x) & \cdots & D_{\tilde{n}\tilde{n}}(x) \end{pmatrix},
\]	
and 
\[
V(x)=\begin{pmatrix} V_{11}(x) & 0 & \cdots & 0 \\ V_{21}(x) & V_{22}(x) & \cdots & 0 \\ \vdots & \vdots & \ddots & \vdots \\ V_{\tilde{n}1}(x) & V_{\tilde{n}2}(x) & \cdots & V_{\tilde{n}\tilde{n}}(x) \end{pmatrix}, \quad  F(x)=\begin{pmatrix} F_{11}(x) & 0 & \cdots & 0 \\ F_{21}(x) & F_{22}(x) & \cdots & 0 \\ \vdots & \vdots & \ddots & \vdots \\ F_{\tilde{n}1}(x) & F_{\tilde{n}2}(x) & \cdots & F_{\tilde{n}\tilde{n}}(x) \end{pmatrix},
\]
where  the size of the $k$-th diagonal block is still denoted by $m_{k}$. 
For $\mu>0$, define the operator $ \mathcal{L}_{\mu} : X \to X$ by 
\[
[\mathcal{L}_{\mu} \phi](x) := D(x)[\mathcal{J}\phi](x) + V(x)\phi(x) + \frac{1}{\mu} F(x)\phi(x), \quad  \phi \in X, x\in \overline{\Omega},
\]
Similarly, we define $\mathcal{L}_{\mu,k}$ analogously to $L_{k}$ in Proposition \ref{pr3.1}.
Define the operators $\mathcal{L}_{\infty} ,\mathcal{F} : X\to X$ by
\[
\left[ \mathcal{L}_{\infty} \phi\right] (x) := D(x) \left[ \mathcal{J}u\right](x) + V(x)\phi(x), \quad \left[ \mathcal{F}\phi\right] (x):= F(x)\phi(x), \quad \phi\in X.
\]
By the general theory of semigroups, the system \eqref{eq5.2} generates  a $C_0$-semigroup $\Phi(t)$ on $X$ such that $\left[ \Phi(t)u_{0}\right](x)= u(t,x)$, where $u(t,x)$ denotes the solution to system \eqref{eq5.2} with the initial condition $u(0, x)=u_0(x)$. Define the operators $Q: X\to X$ and $\hat Q: X \to X $ by
\[
\left[ Qu\right] (x):= \int_{0}^{+\infty}F(x)\left[ \Phi(t)u\right](x) \, \mathrm{d}t, \quad u \in X,
\] 
and 
\[
\left[ \hat Qu\right] (x):= \int_{0}^{+\infty}\left[ \Phi(t)\mathcal{F}u\right](x) \, \mathrm{d}t, \quad u \in X.
\]
By the Gelfand's formula (see, e.g., \cite[Theorem
VI.6]{MR751959}), it is easy to verify that $r(Q)=r(\hat Q)$.
Clearly, $Q= -\mathcal{F}\mathcal{L}_{\infty}^{-1}$; see \cite[Theorem 3.12]{MR2505085}.
The spectral radius of $Q$ or $\hat Q$ is then defined as the basic reproduction ratio for system \eqref{eq4.1}, i.e., 
\[
\mathcal{R}_{0}= r(Q) \text{ or } \mathcal{R}_{0} =r (\hat Q).
\]

To apply the theory of the basic reproduction ratio, as developed in \cite{MR3032845,MR3612966,MR3992071,MR2505085}, we impose the following assumption:
\begin{itemize}
	\item [(\textbf{L})] $s(\mathcal{L}_\infty)<0$.
\end{itemize}

According to  \cite[Theorem 3.5]{MR2505085} and \cite[Theorem 3.4]{lin2026}, we have the following conclusions.	

\begin{proposition}\label{pr4.1}
	The following statements are valid:
	\begin{itemize}
		\item [\rm(i)] 	$s(\mathcal{L}_{\mu})$ is continuous with respect to $\mu>0$;
		
		\item [\rm(ii)] If $s( \mathcal{L}_{\infty}) <0$, then $\mathcal{R}_{0}-1 $ has the same sing as $s(\mathcal{L}_{1})$; moreover, $\frac{1}{\mu}r(-\mathcal{F} \mathcal{L}_{\infty}^{-1}) -1$ has the same sing as  $s(\mathcal{L}_{\mu})$;
		
		\item[\rm(iii)] There exists a unique $\mu_{0}>0$ such that $s(\mathcal{L}_{\mu_{0}})=0 $ if and only if    $s(\mathcal{L}_{\infty})<0$; moreover, $\mu_{0}=\mathcal{R}_{0}$.
	\end{itemize}
\end{proposition}

\begin{proof}
	
	By Proposition \ref{pr3.2}, we have $s(\mathcal{L}_{\mu}) = \max\limits_{1\le k\le \tilde{n}}s(\mathcal{L}_{\mu,k})$. Recall that, by \cite[Proposition 3.1]{lin2026}, $s(\mathcal{L}_{\mu,k})$  is continuous with respect to $\mu>0$. Therefore, $s(\mathcal{L}_{\mu})$ is continuous with respect to $\mu>0$.
	
	Moreover, statements (ii) and (iii) follow directly from \cite[Theorem 3.5]{MR2505085} and \cite[Theorem 3.4]{lin2026}, respectively.
\end{proof}

Note that, in contrast to \cite[Proposition 3.1]{lin2026}, we obtain Proposition \ref{pr4.1}(i) without the weakly irreducible condition on $V+F$.

For each $1\le k\le \tilde{n}$, set 
\[
\mathcal{X}_{k}:= \left\lbrace  \phi \in \mathrm{Int}(X_{k}^{+}) : \sum_{j\in \mathcal{S}_{k}}-d_{ij}(x) \int_{\Omega} J_{j}(x, y) \phi_j(y) \,\mathrm{d}y - \sum_{j\in \mathcal{S}_{k}}  v_{ij}(x)\phi_{j}(x) >0, \forall x\in \overline\Omega, i\in \mathcal{S}_{k}\right\rbrace. 
\]
For each $1\le k\le \tilde{n}$, $\phi \in \mathrm{Int}(X_{k}^{+})$ and $i\in \mathcal{S}_{k}$, set 
\[
\Omega_{\phi,i}^{k}:= \left\lbrace x\in \overline{\Omega}: \sum_{j\in \mathcal{S}_{k}}-d_{ij}(x) \int_{\Omega} J_{j}(x, y) \phi_j(y) \,\mathrm{d}y - \sum_{j\in \mathcal{S}_{k}}  v_{ij}(x)\phi_{j}(x) >0 \right\rbrace .
\]

\begin{lemma}\label{le4.2}
  Assume that there exists a unique $\mu_{0}>0$ such that $s(\mathcal{L}_{\mu_{0}})=0 $. Then 
  \[
  \begin{aligned}
  	\mu_{0}= & \max\limits_{1\le k\le \tilde{n}}\inf_{\phi \in \mathcal{X}_{k} } \sup_{x \in \Omega, i \in \mathcal{S}_{k}} \dfrac{\sum_{j\in \mathcal{S}_{k}}f_{ij}(x)\phi_{j}(x)}{\sum_{j\in \mathcal{S}_{k}}-d_{ij}(x) \int_{\Omega} J_{j}(x, y) \phi_j(y) \,\mathrm{d}y - \sum_{j\in \mathcal{S}_{k}}  v_{ij}(x)\phi_{j}(x)}
  	\\[0.4em]
  	=& \max\limits_{1\le k\le \tilde{n}} \sup_{\phi \in \mathrm{Int}(X_{k}^{+}) } \inf_{(x, i)\in (\Omega_{\phi,i}^{k}, \mathcal{S}_{k} )}  \dfrac{\sum_{j\in \mathcal{S}_{k}}f_{ij}(x)\phi_{j}(x)}{\sum_{j\in \mathcal{S}_{k}}-d_{ij}(x) \int_{\Omega} J_{j}(x,y) \phi_j(y) \,\mathrm{d}y - \sum_{j\in \mathcal{S}_{k}}  v_{ij}(x)\phi_{j}(x)}.
  \end{aligned}
  \]
\end{lemma}

\begin{proof}
By Proposition \ref{pr4.1}, we know that $s( \mathcal{L}_{\infty})<0$, and hence  $s( \mathcal{L}_{\infty,k})<0$ where $\mathcal{L}_{\infty,k}$  is defined analogously to $\mathcal{L}_{\infty}$. According to \cite[Theorem 3.4]{lin2026}, for each $k$, there exists a unique  $\mu_{0,k}$ such that $s(\mathcal{L}_{\mu_{0,k},k})=0$, and 
\[
\begin{aligned}
	\mu_{0,k}= &\inf_{\phi \in \mathcal{X}_{k} }  \sup_{x \in \Omega, i \in \mathcal{S}_{k}} \dfrac{\sum_{j\in \mathcal{S}_{k}}f_{ij}(x)\phi_{j}(x)}{\sum_{j\in \mathcal{S}_{k}}-d_{ij}(x) \int_{\Omega} J_{j}(x- y) \phi_j(y) \,\mathrm{d}y - \sum_{j\in \mathcal{S}_{k}}  v_{ij}(x)\phi_{j}(x)}
	\\[0.4em]
	=& \sup_{\phi \in \mathrm{Int}(X_{k}^{+}) } \inf_{(x, i)\in (\Omega_{\phi,i}^{k}, \mathcal{S}_{k} )} \dfrac{\sum_{j\in \mathcal{S}_{k}}f_{ij}(x)\phi_{j}(x)}{\sum_{j\in \mathcal{S}_{k}}-d_{ij}(x) \int_{\Omega} J_{j}(x- y) \phi_j(y) \,\mathrm{d}y - \sum_{j\in \mathcal{S}_{k}}  v_{ij}(x)\phi_{j}(x)}.
\end{aligned}
\]
Since $s(\mathcal{L}_{\mu})= \max\limits_{1\le k\le \tilde{n}} s(\mathcal{L}_{\mu,k})$ and each $s(\mathcal{L}_{\mu,k})$ is nonincreasing in $\mu$, we obtain $\mu_{0}= \max\limits_{1\le k\le \tilde{n}} \mu_{0,k}$, which yields the desired conclusion.
\end{proof}

From Proposition \ref{pr4.1} and Lemma \ref{le4.2}, we  derive the main results in this section as follows:

\begin{theorem}
	Assume that {\rm(\textbf{L})} holds. Then $\mathcal{R}_{0}-1$ has the same sing as $s(\mathcal{L}_{1})$, and $ \mu=\mathcal{R}_{0}$ is the unique positive solution of $s(\mathcal{L}_{\mu})=0$. Moreover, 
	\[
	\begin{aligned}
		\mathcal{R}_{0}= & \max\limits_{1\le k\le \tilde{n}}\inf_{\phi \in \mathcal{X}_{k} }  \sup_{x \in \Omega, i \in \mathcal{S}_{k}} \dfrac{\sum_{j\in \mathcal{S}_{k}}f_{ij}(x)\phi_{j}(x)}{\sum_{j\in \mathcal{S}_{k}}-d_{ij}(x) \int_{\Omega} J_{j}(x, y) \phi_j(y) \,\mathrm{d}y - \sum_{j\in \mathcal{S}_{k}}  v_{ij}(x)\phi_{j}(x)}
		\\[0.4em]
		=& \max\limits_{1\le k\le \tilde{n}} \sup_{\phi \in \mathrm{Int}(X_{k}^{+}) } \inf_{(x, i)\in (\Omega_{\phi,i}^{k}, \mathcal{S}_{k} )} \dfrac{\sum_{j\in \mathcal{S}_{k}}f_{ij}(x)\phi_{j}(x)}{\sum_{j\in \mathcal{S}_{k}}-d_{ij}(x) \int_{\Omega} J_{j}(x,y) \phi_j(y) \,\mathrm{d}y - \sum_{j\in \mathcal{S}_{k}}  v_{ij}(x)\phi_{j}(x)}.
	\end{aligned}
	\]
\end{theorem}

\section{Dynamics of stem cell regeneration models}

In this section, we apply the established spectral theory to study the dynamical behavior of a class of multiple‑genotype stem cell regeneration models with epigenetic transitions, with or without gene mutations.

This problem has been previously investigated by Su et al.  \cite{MR4612706} under the assumption that an associated operator possesses a principal eigenvalue.
Here, we remove this assumption. The absence of a principal eigenvalue introduces a fundamental difficulty in analyzing nonlocal dispersal dynamics.  
By applying the more general theory  of principal eigenvalues developed in Section 3, we overcome this difficulty and derive the same spectral conclusions. Consequently, the dynamical behavior of the system can be established by following the arguments in \cite{MR4612706}.

Furthermore, we extend the analysis to the reducible case, which was not considered in \cite{MR4612706}, and study its dynamical behavior.
Meanwhile, we will employ the characterization of the spectral bound established in Section 3 to provide a quantitative description of the key parameters in the model, and perform numerical simulations accordingly.

\subsection{The irreducible case}

In this subsection, we extend the results of \cite{MR4612706} to the case where the principal eigenvalue does not exist, with the system irreducible.

The system is given by
\begin{equation}\label{eq4.3}
	\begin{cases}
		\begin{aligned}
			\frac{\partial u_i(x,t)}{\partial t} &= 2(1+\epsilon m_{ii})\beta(s(t))\int_{\Omega} p_i(x,y)e^{-\mu_i(y)}u_i(y,t)\, \mathrm{d}y \\
			&\quad + 2\epsilon\sum_{j\neq i} m_{ij}\beta(s(t))\int_{\Omega} p_j(x,y)e^{-\mu_j(y)}u_j(y,t)\, \mathrm{d}y \\
			&\quad - \beta(s(t))u_i(x,t) - \kappa(x)u_i(x,t), \quad (x,t)\in\overline{\Omega}\times(0,\infty),\ 1\leq i\leq n,
		\end{aligned}\\[3em]
	\displaystyle	s(t) = \sum\limits_{i=1}^n \int_{\Omega} \zeta(x) u_i(x,t) \, \mathrm{d}x, \hspace{6.6cm} t\in[0,\infty),\\[1em]
		u_i(x,0) = u_{i,0}(x), \hspace{7.3cm} x\in\overline{\Omega},\ 1\leq i\leq n.
	\end{cases}
\end{equation}
Here  $\Omega$ is a bounded domain in $\mathbb{R}^{m}$ with smooth  boundary. The function  $u_{i}(x,t)$ denote the density of stem cells of genetic type $i$ in the resting phase, with epigenetic state $x $ and time $t$.  Cells leave the resting phase at differentiation rate $\kappa(x)$ and  enter proliferation at  proliferation rate $\beta(s(t))$. During division, a cell of type $j$ and state $y$ survives apoptosis with probability $e^{-\mu_j(y)}$ and produces two daughter cells. The daughter inherits the same genetic type $i=j$ with probability $1+\epsilon m_{ii}$, or mutates from $j$ to $i\neq j$ with probability $\epsilon m_{ij}$, and $m_{ii}=-\sum\limits_{j\neq i}m_{ji}$. The epigenetic state may change from $y$ to $x$ according to the kernel $p_i(x,y)$. The function $s(t)$ is the total cytokine concentration, with secretion rate $\zeta(x)$.

For more details, including the model derivation, we refer to \cite{MR4070771,MR4067996,MR4278592,MR4612706}. 
To analyze dynamics of system  \eqref{eq4.3}, we make the following assumption:
\begin{itemize}

	\item [(\textbf{A1})]  $M=(m_{ij})_{n \times n}$ is  a  cooperation matrix, $ 0 < \epsilon \leq \epsilon_0 = \min_{1 \leq i \leq n} \{-\frac{1}{m_{ii}}\}; \, p_i \in C(\overline{\Omega} \times \overline{\Omega}) $  is nonnegative with $ p_i(x, x) > 0 $ for all $ x \in \overline{\Omega}$ and $\int_{\Omega}p_{i}(x,y)\, \mathrm{d}x=1$ for all $ y \in \overline{\Omega}$;  $ \mu_i, \kappa, \kappa_{i} \in C(\overline{\Omega}) $ with $ \mu_i \geq 0 $ and $ \kappa, \kappa_{i} > 0$;
	
	\item [(\textbf{A2})]   $\beta, \beta_{i} \in C^1([0, \infty))$ are bounded and strictly decreasing with a positive lower bound, and $\beta'$ and $\beta_{i}$ are also bounded on $[0, \infty)$;

	\item [(\textbf{A3})] $ M := (m_{ij})_{n \times n}$  is irreducible.

\end{itemize}

We first study a more general eigenvalue problem associated with the system, and based on this, analyze the dynamical behavior of the system. To this end, we define the  operator $\mathcal{P} : X \to X  $ by
\[
\mathcal{P}\varphi := 
\left( 	\mathcal{P}_1 \varphi_1, 
	\mathcal{P}_2 \varphi_2, 
  \dots,
	\mathcal{P}_n \varphi_n\right) ^{T}, \quad \varphi \in X,
\]
with 
\[
\left[ \mathcal{P}_i\varphi_i\right] (x) := 2 \int_{\Omega} p_i(x, y)e^{-\mu_i(y)}\varphi_i(y) \, \mathrm{d}y, \quad  1\leq i \leq n,
\]
and $ \mathcal{H} : X \to X $ by
\[
\left[ \mathcal{H}\varphi\right] (x) := (h_1(x)\varphi_1(x), h_2(x)\varphi_2(x), \dots, h_n(x)\varphi_n(x))^T, \quad \varphi \in X,
\]
where $ h_i(x) := -1 - \alpha_i(\lambda)\kappa_i(x) $  with  $\alpha_{i}(\lambda):= \frac{1}{\beta_{i}(\lambda)} $  for all $ 1 \leq i \leq n $. 
Denote 
\[
\mathcal{T}_{\epsilon,\lambda}:= \left( I+\epsilon M\right) \mathcal{P} + \mathcal{H}.
\]

\begin{lemma}\label{le6.1}
	 Assume that {\rm(\textbf{A1})}--{\rm(\textbf{A3})} hold. The following conclusions hold:
	 
	 \begin{itemize}
	 	\item [\rm(i)] For each fixed $  0 < \epsilon \leq \epsilon_0 $, $s(\mathcal{T}_{\epsilon,\lambda})$ is continuous and strictly decreasing  in $\lambda\in \left[ 0, \infty\right) $;

	 	\item [\rm(ii)]  
	 	\[
	 	\lim\limits_{\epsilon\to 0^{+}} s(\mathcal{T}_{\epsilon,\lambda}) = \gamma(\lambda):=\max\limits_{1\le i\le n} s(\mathcal{T}_{\lambda}^{i}),
	 	\]
	 	where $\mathcal{T}_{\lambda}^{i}: C(\overline \Omega) \to C(\overline \Omega)$ is defined by
	 	\[
	 	\left[ \mathcal{T}_{\lambda}^{i}  \phi \right](x) := \left[ \mathcal{P}_i \phi\right] (x) + h_{i}(x)\phi(x), \quad x\in \overline \Omega.
	 	\]
	 \end{itemize}
\end{lemma}

\begin{proof}
(i) Let $\lambda_{0}\in [0, \infty)$ be fixed. By Lemma  \ref{le2.2}, for any $\delta > 0$, there exist $\phi \in \operatorname{Int} (X^{+})$ and  $\varphi \in \operatorname{Int} (X^{+})$ such that
\[
\mathcal{T}_{\epsilon,\lambda_{0}} \phi \le s(\mathcal{T}_{\epsilon,\lambda_{0}}) \phi + \frac{1}{2}\delta \phi, \qquad 
\mathcal{T}_{\epsilon,\lambda_{0}} \varphi \ge s(\mathcal{T}_{\epsilon,\lambda_{0}}) \varphi - \frac{1}{2}\delta \varphi.
\]
Denote $\overline \kappa = \max\limits_{1\le i\le n} \max\limits_{x\in\overline{\Omega}}\kappa_{i}(x)$. For such  $\delta$, there exists $\delta^{\ast}$ such that for all $\left|\lambda - \lambda_{0}\right|\le\delta^{\ast} $, 
\[ 
\left| \alpha_{i}(\lambda)-\alpha_{i}(\lambda_{0}) \right| \le \frac{\delta}{2 \overline \kappa} \quad  \text{ for } 1\le i\le n.
\] 
Then, for  $\left|\lambda - \lambda_{0}\right|\le\delta^{\ast} $, 
\[
\mathcal{T}_{\epsilon,\lambda} \phi \le s(\mathcal{T}_{\epsilon,\lambda_{0}}) \phi + \delta \phi, \qquad 
 \mathcal{T}_{\epsilon,\lambda} \varphi \ge s(\mathcal{T}_{\epsilon,\lambda_{0}}) \varphi - \delta \varphi.
\]
It follows that $\left| s(\mathcal{T}_{\epsilon,\lambda})-s(\mathcal{T}_{\epsilon,\lambda_{0}})\right| \le \delta$ whenever  $\left|\lambda - \lambda_{0}\right|\le\delta^{\ast} $.

It is easy to verify that $ s(\mathcal{T}_{\epsilon,\lambda})$ is nonincreasing in $\lambda$. Assume, for contradiction, that there exists $\lambda_{1}>\lambda_{2}$ such that $s(\mathcal{T}_{\epsilon,\lambda_{1}})=s(\mathcal{T}_{\epsilon,\lambda_{2}}) $. Since each $\alpha_i $ strictly increasing for $1\le i \le n$, we have 
\[ 
\max\limits_{1\le i\le n} \max\limits_{x\in\overline{\Omega}}\left\lbrace -1- \alpha_{i}(\lambda_{1}) \kappa_{i}(x) \right\rbrace  < \max\limits_{1\le i\le n} \max\limits_{x\in\overline{\Omega}}\left\lbrace -1- \alpha_{i}(\lambda_{2}) \kappa_{i}(x) \right\rbrace\le s(\mathcal{T}_{\epsilon,\lambda_{2}}) = s(\mathcal{T}_{\epsilon,\lambda_{1}}).
\] 
By Proposition \ref{pr2.1}, there exists $\psi\in$ Int$(X^{+})$ such that $\mathcal{T}_{\epsilon,\lambda_{1}} \psi =  s(\mathcal{T}_{\epsilon,\lambda_{1}}) \psi$. Then 
\[
\mathcal{T}_{\epsilon,\lambda_{2}} \psi \ge  s(\mathcal{T}_{\epsilon,\lambda_{1}}) \psi + c_{0} \psi,
\]
where $c_{0}=\min\limits_{1\le i\le n}\min\limits_{x\in\overline{\Omega}} \left\lbrace  \alpha_{i}(\lambda_{1})\kappa_{i}(x)- \alpha_{i}(\lambda_{2})\kappa_{i}(x)\right\rbrace>0$. Hence $s(\mathcal{T}_{\epsilon,\lambda_{2}})> s(\mathcal{T}_{\epsilon,\lambda_{1}})$, a contradiction.

(ii) Define $\hat{\mathcal{T}}_{\lambda}: X \to X $ by 
\[
\hat{\mathcal{T}}_{\lambda} \phi = (\mathcal{T}_{\lambda}^{1}  \phi_{1}, \dots, \mathcal{T}_{\lambda}^{n}  \phi_{n})^{T}. 
\]
By the  perturbation theory for linear operators(see, e.g., \cite[Section IV.3]{MR203473}), then
\begin{equation}\label{eq6.1}
\limsup\limits_{\epsilon\to 0^{+}} s(\mathcal{T}_{\epsilon,\lambda}) \le  s(\hat{\mathcal{T}}_{\lambda}).
\end{equation}
If $m_{ii}\ge 0$, then $s(\mathcal{T}_{\epsilon,\lambda}) \ge  s(\hat{\mathcal{T}}_{\lambda})$ for any $\epsilon>0$. 
Consequently,
\[
\lim\limits_{\epsilon\to 0^{+}} s(\mathcal{T}_{\epsilon,\lambda}) =  s(\hat{\mathcal{T}}_{\lambda})=\max\limits_{1\le i\le n} s(\mathcal{T}_{\lambda}^{i}).
\]

Next, we discuss the case where $m_{ii}<0$.
Define $\hat{\mathcal{T}}_{\epsilon,\lambda}^{i} : C(\overline{\Omega}) \to C(\overline{\Omega})$ by 
\[
\left[ \hat{\mathcal{T}}_{\epsilon,\lambda}^{i} \phi\right](x)  = \left( 1+ \epsilon m_{ii}\right) \left[ \mathcal{P}_i \phi \right](x) + h_{i}(x)\phi(x), \quad x\in \overline{\Omega},
\] 
and set $\hat{\mathcal{T}}_{\epsilon,\lambda}= (\hat{\mathcal{T}}_{\epsilon,\lambda}^{1},\dots, \hat{\mathcal{T}}_{\epsilon,\lambda}^{n})^{T} $.
It is easy to verify that 
\begin{equation}\label{eq6.2}
	s({\mathcal{T}}_{\epsilon,\lambda})\ge s(\hat{\mathcal{T}}_{\epsilon,\lambda}). 
\end{equation}
For any $\delta>0$, there exists $\phi\in$Int$(X^{+})$ such that
\[
{\mathcal{T}}_{\lambda}^{i}\phi_i\ge s({\mathcal{T}}_{\lambda}^{i})\phi_i -\frac{\delta}{2}\phi_i, \qquad  1\le i\le n.
\]
Then 
\[
\hat{\mathcal{T}}_{\epsilon,\lambda}^{i} \phi_{i}\ge s({\mathcal{T}}_{\lambda}^{i})\phi_i -\frac{\delta}{2}\phi_i + \epsilon m_{ii} \phi_{i}\max\limits_{x \in \overline{\Omega}}\frac{\left[ \mathcal{P}_i \phi_{i}\right] (x) }{\phi_{i}(x)}, \qquad  1\le i\le n.
\]
It follows that 
\[
\liminf\limits_{\epsilon\to 0^{+}} s(\hat{\mathcal{T}}_{\epsilon,\lambda}^{i}) \ge s({\mathcal{T}}_{\lambda}^{i}) - \delta, \qquad  1\le i\le n.
\]
Since $\delta>0$ is arbitrary, we obtain $\liminf\limits_{\epsilon\to 0^{+}} s(\hat{\mathcal{T}}_{\epsilon,\lambda}^{i}) \ge s({\mathcal{T}}_{\lambda}^{i})$ for each $ 1\le i\le n$. This implies
\begin{equation}\label{eq6.3}
\liminf\limits_{\epsilon\to 0^{+}} s(\hat{\mathcal{T}}_{\epsilon,\lambda}) \ge s(\hat{\mathcal{T}}_{\lambda}).
\end{equation}
Combining \eqref{eq6.1}, \eqref{eq6.2}, and \eqref{eq6.3} yields the desired conclusion.
\end{proof}

\begin{theorem}\label{th6.1}
	Assume that {\rm(\textbf{A1})} and {\rm(\textbf{A2})} hold. The following conclusions hold:
	
	\begin{itemize}
		\item [\rm(i)] For each fixed $  0 < \epsilon \leq \epsilon_0 $, $s(\mathcal{T}_{\epsilon,\lambda})$ is continuous and strictly decreasing  in $\lambda\in \left[ 0, \infty\right) $;

		\item [\rm(ii)]  
		\[
		\lim\limits_{\epsilon\to 0^{+}} s(\mathcal{T}_{\epsilon,\lambda}) = \gamma(\lambda):=\max\limits_{1\le i\le n} s(\mathcal{T}_{\lambda}^{i}),
		\]
		where $\mathcal{T}_{\lambda}^{i}: C(\overline \Omega) \to C(\overline \Omega)$ is defined by
		\[
		\left[ \mathcal{T}_{\lambda}^{i}  \phi \right](x) := \left[ \mathcal{P}_i \phi\right] (x) + h_{i}(x)\phi(x), \quad x\in \overline \Omega.
		\]
	\end{itemize}
\end{theorem}

\begin{proof}
	If $M$ is irreducible, the conclusion has already been proved in Lemma \ref{le6.1}. If $M$ is reducible, then by  \cite[section 2.3]{MR1298430}, there exists a permutation matrix $P$ such that 
	\[
	PMP^{T}=\begin{pmatrix} M_{11}& 0 & \cdots & 0 \\ M_{21} & M_{22} & \cdots & 0 \\ \vdots & \vdots & \ddots & \vdots \\ M_{\tilde{n}1}& M_{\tilde{n}2} & \cdots & M_{\tilde{n}\tilde{n}}\end{pmatrix},
	\]
	where each $A_{ii}$ is irreducible for $1\le i \le \tilde n$. Combining Proposition \ref{pr3.2} with Lemma \ref{le6.1} then yields the desired conclusions. 
\end{proof}

\begin{remark}
Compare to 	 \cite[Theorem 1.1(ii)(iii)]{MR4612706}, Theorem \ref{th6.1} removes the key assumption
\begin{equation}\label{eq4.5}
\frac{1}{\kappa_i - \min_{\Omega} \kappa_i} \notin L^1(\Omega) \quad \text{for all } 1 \leq i \leq n,
\end{equation}
that guarantees the existence of the principal eigenvalue, while still yielding the same spectral conclusions. Namely, we obtain results on the spectral bound even when the principal eigenvalue does not exist. This situation better reflects the nature of nonlocal problems, where a principal eigenvalue may be absent. Furthermore,  we extend the results to the case where $M$ is reducible.
\end{remark}

\begin{remark}
The conclusions of Theorem \ref{th6.1} also holds for the more general case without requiring the condition $m_{ii}=-\sum\limits_{j\neq i}m_{ji}$.
\end{remark}

Next, we analyze the steady states of the system \eqref{eq4.3}. We first consider the case without gene mutations, i.e., $\epsilon=0$. Consider the more general steady state equation
\begin{equation}\label{eq4.7}
	\begin{cases} 
		\displaystyle 2 \int_{\Omega} p(x, y) e^{-\mu(y)} \theta_{i}(y)\, \mathrm{d}y - \theta_{i}(x) - \alpha_{i}(s) \kappa_{i}(x) \theta_{i}(x) = 0, & x \in \overline{\Omega}, 1\le i \le n, \\[1em] 
		\displaystyle s = \sum\limits_{i = 1}^{n}\int_{\Omega} \zeta(x) \theta_{i}(x) \, \mathrm{d}x.
	\end{cases}
\end{equation}
In this case, the components of the system are not coupled,  and we consider the following equation for each component:
\begin{equation}\label{eq4.4}
 \begin{cases} 
 	\displaystyle 2 \int_{\Omega} p(x, y) e^{-\mu(y)} \theta(y)\, \mathrm{d}y - \theta(x) - \alpha_{i}(s_{i}) \kappa_{i}(x) \theta(x) = 0, & x \in \overline{\Omega}, \\[1em] 
 	\displaystyle s_{i} = \int_{\Omega} \zeta(x) \theta(x) \, \mathrm{d}x,
 \end{cases}
\end{equation}
and its associated eigenvalue problem $\mathcal{T}_{\lambda}^{i}$: 
\[
\left[ {\mathcal{T}}_{\lambda}^{i}\phi\right] (x)=\left[ \mathcal{P}_i \phi\right] (x) -\phi(x)- \alpha_{i}(\lambda)\kappa_{i}(x)\phi(x), \quad x\in \overline{\Omega}. 
\]
To simplify, we replace $\alpha_{i}$ by $\lambda$ in $\mathcal{T}_{\lambda}^{i}$ and denote denote the resulting operator by $\tilde{\mathcal{T}}_{\lambda}^{i}$.
Denote $Y:=C(\overline{\Omega})$ and $Y^{+}:=(\overline\Omega, \mathbb{R_{+}})$.

\begin{proposition}\label{pr4.3}
 The following statements valid:
 
 \begin{itemize}
 	\item [\rm(i)] If  $s( \mathcal{P}_i)\le1$, then 
 $s(\tilde{\mathcal{T}}_{\lambda}^{i})\le 0$ for all $ \lambda>0$. 
 
 \item [\rm(ii)] If $s( \mathcal{P}_i)>1$, then there exists a unique $\lambda^{\ast}_{i}>0$ such that $s(\tilde{\mathcal{T}}_{\lambda^{\ast}_{i}}^{i})=0$. Moreover,  $\tilde{\mathcal{T}}_{\lambda^{\ast}_{i}}^{i}$ admits the principal eigenvalue, and 
 \[
 \begin{aligned}
 	\lambda^{\ast}_{i}= & \inf_{\phi \in \mathrm{Int}(Y^{+})}\sup_{x \in \Omega}\dfrac{2 \int_{\Omega} p_i(x, y)e^{-\mu_i(y)}\phi(y) \, \mathrm{d}y -\phi(x)}{\kappa_{i}(x)\phi(x)}\\
 	= & \sup_{\phi \in  \mathrm{Int}(Y^{+})} \inf_{x \in \Omega} \dfrac{2 \int_{\Omega} p_i(x, y)e^{-\mu_i(y)}\phi(y) \, \mathrm{d}y -\phi(x)}{\kappa_{i}(x)\phi(x)}.
 \end{aligned}
 \]
\end{itemize} 	
\end{proposition}

\begin{proof}
(i) By Theorem \ref{th6.1}(i), the conclusion follows.

(ii) By the  perturbation theory for linear operators(see, e.g., \cite[Section IV.3]{MR203473}), there exists $\lambda_{0}>0$ such that $s(\tilde{\mathcal{T}}_{\lambda_{0}}^{i})>0$. Similarly, there exists $\lambda_{1}>0$ such that $s(\frac{1}{\lambda_{1}}\tilde{\mathcal{T}}_{\lambda_{1}}^{i})<0$, which implies $s(\tilde{\mathcal{T}}_{\lambda_{1}}^{i})<0$. According to Theorem \ref{th6.1}(i),  there exists a unique $\lambda^{\ast}_{i}$ satisfying  $s(\tilde{\mathcal{T}}_{\lambda^{\ast}_{i}}^{i})=0$.
Notice that $\max\limits_{x \in \overline{\Omega}} \left\lbrace -1 - \lambda^{\ast}_{i} \kappa_{i}(x)\right\rbrace <0 = s(\tilde{\mathcal{T}}_{\lambda^{\ast}_{i}}^{i})$. By Proposition \ref{pr2.1},  $\tilde{\mathcal{T}}_{\lambda^{\ast}_{i}}^{i}$ admits the principal eigenvalue.

Set 
\[
\lambda_{1}=\inf_{\phi \in \mathrm{Int}(Y^{+})}\sup_{x \in \Omega}\dfrac{2 \int_{\Omega} p_i(x, y)e^{-\mu_i(y)}\phi(y) \, \mathrm{d}y -\phi(x)}{\kappa_{i}(x)\phi(x)}, 
\]
and 
\[
\lambda_{2}=\sup_{\phi \in  \mathrm{Int}(Y^{+})} \inf_{x \in \Omega} \dfrac{2 \int_{\Omega} p_i(x, y)e^{-\mu_i(y)}\phi(y) \, \mathrm{d}y -\phi(x)}{\kappa_{i}(x)\phi(x)}.
\]
It is easy to see that 
\[
\lambda_{1}= \inf \left\lbrace \lambda>0: \exists \phi \in \mathrm{Int}(Y^{+})\ s.t. \ \tilde{\mathcal{T}}_{\lambda}^{i} \phi \le 0\right\rbrace ,
\]
and 
$\lambda_{2}= \sup  \left\lbrace \lambda>0: \exists \phi \in \mathrm{Int}(Y^{+})\  s.t. \ \tilde{\mathcal{T}}_{\lambda}^{i} \phi \ge 0\right\rbrace $.
Since $s(\tilde{\mathcal{T}}_{\lambda}^{i})$ is  strictly decreasing in $\lambda$, then 
\[
s(\tilde{\mathcal{T}}_{\lambda}^{i}) <0 \text{ for } \lambda> \lambda_{i}^{\ast}, \qquad  s(\tilde{\mathcal{T}}_{\lambda}^{i}) >0 \text{ for } \lambda<\lambda_{i}^{\ast}.
\]
By Lemma \ref{le2.2}, for any $\lambda>\lambda_{i}^{\ast}$, there exists $\psi\in\mathrm{Int}(Y^{+})$ such that  $\tilde{\mathcal{T}}_{\lambda}^{i} \psi\le 0 $.
Consequently, $\lambda_{1}\le \lambda_{i}^{\ast}$. If $\lambda_{1}< \lambda_{i}^{\ast} $, then by the definition of $ \lambda_{1}$, there exist $\lambda^{'}\in (\lambda_{1}, \lambda_{i}^{\ast})$ and $\psi\in \mathrm{Int}(Y^{+})$ such that $\tilde{\mathcal{T}}_{\lambda^{'}}^{i} \psi\le 0 $, which yields $s(\tilde{\mathcal{T}}_{\lambda^{'}}^{i})\le 0$, a contradiction. Hence, $ \lambda_{1}=\lambda_{i}^{\ast}$.

Similarly, for any $\lambda<  \lambda_{i}^{\ast}$, there exists $\psi\in \mathrm{Int}(Y^{+})$ such that $\tilde{\mathcal{T}}_{\lambda}^{i} \psi> 0 $, which gives $\lambda_{i}^{\ast} \le \lambda_{2} $. If $\lambda_{i}^{\ast} < \lambda_{2} $, then there exist $\lambda^{'}\in (\lambda_{i}^{\ast},\lambda_{2})$ and $\psi\in \mathrm{Int}(Y^{+})$ such that $\tilde{\mathcal{T}}_{\lambda^{'}}^{i} \psi\ge 0 $, implying  $s(\tilde{\mathcal{T}}_{\lambda^{'}}^{i})\ge 0$, again a contradiction. Thus  $ \lambda_{2}=\lambda_{i}^{\ast}$.
\end{proof}

\begin{proposition}\label{pr4.2}
	 Assume that   {\rm(\textbf{A1})} and {\rm(\textbf{A2})} hold. The following conclusions hold:
	\begin{itemize}
		\item [\rm(i)] If $s( \mathcal{T}_{0}^{i})<0$ or $ s( \mathcal{T}_{+\infty}^{i})  > 0$, then $0$ is the only nonnegative solution of \eqref{eq4.4} in $C(\overline{\Omega})$;
		
		\item [\rm(ii)] If $s( \mathcal{T}_{+\infty}^{i}) < 0 < s( \mathcal{T}_{0}^{i})$, then \eqref{eq4.4}  has a unique positive solution, denoted by $\theta_{i}$.

	\end{itemize}
\end{proposition}

Note that Proposition \ref{pr4.2} can be proved by  Proposition \ref{pr4.3}(ii)  and the arguments in \cite[Theorem 3.1]{MR4612706}, so we omit it here. It should be pointed out that $\theta_{i}$ exists if and only if $\tilde{\mathcal{T}}_{\lambda^{\ast}_{i}}^{i}$ admits a principal eigenvalue, with associated principal eigenfunction denoted by $\varphi_{i}$; moreover,
\[
\theta_{i}(x) = \frac{s_{i} \varphi_{i}(x)}{\int_{\Omega} \zeta(x) \varphi_{i}(x) \, \mathrm{d}x},
\]
where $s_{i} = \alpha_{i}^{-1}(\lambda^{\ast}_{i})$. Here, we remove condition \eqref{eq4.5} in \cite[Theorem 3.1]{MR4612706} and still obtain the same results, since $\tilde{\mathcal{T}}_{\lambda^{\ast}_{i}}^{i}$ admits a principal eigenvalue by Proposition \ref{pr4.3}(ii).

Combining Propositions \ref{pr4.3} and \ref{pr4.2}, we have the following corollary.

\begin{corollary}\label{c4.1}
	 Assume that   {\rm(\textbf{A1})} and {\rm(\textbf{A2})} hold. The following conclusions hold:
	 \begin{itemize}
	 	\item [\rm(i)] If $s( \mathcal{P}_i)\le1$, then $0$ is the only nonnegative solution of \eqref{eq4.4} in $C(\overline{\Omega})$. 
	 	
	 	\item [\rm(ii)] If $s( \mathcal{P}_i)>1$ and
	 	$\lambda_{i}^{\ast} \in \left(  \alpha( 0) , \alpha(+\infty  )\right)$, then \eqref{eq4.4}  has a unique positive solution. 
	 	
	 	\item [\rm(iii)] If  $s( \mathcal{P}_i)>1$ and $\lambda_{i}^{\ast} \not\in \left[  \alpha( 0) , \alpha(+\infty  )\right]$, then $0$ is the only nonnegative solution of \eqref{eq4.4} in $C(\overline{\Omega})$.
	 \end{itemize}
	
\end{corollary}

Next, we discuss the dynamics of system \eqref{eq4.3} both with and without gene mutations. By applying Theorem \ref{th6.1} and Proposition \ref{pr4.2}, and following the arguments in \cite{MR4612706}, we can establish the dynamics of \eqref{eq4.3}. For completeness and the convenience of the reader, we present the relevant conclusions here, and detailed proofs can be found in \cite{MR4612706}. The main difference from \cite{MR4612706} is the removal of Condition \eqref{eq4.5}, which is no longer required under Theorem \ref{th6.1} and Proposition \ref{pr4.2}.

We first discussion the case without gene mutations.
In view of $n$ eigenvalue problems associated with \eqref{eq4.4} and rearrangement,  
the following cases would be considered:
\begin{itemize}
	\item[\textbf{Case 1}] $s(\mathcal{T}_{0}^{i}) < 0$ for all $1 \le i \le n$;
	
	\item[\textbf{Case 2}] $s(\mathcal{T}_{+\infty}^{i}) > 0$ for all $1 \le i \le n$;
	
	\item[\textbf{Case 3}] $s(\mathcal{T}_{0}^{i}) < 0$ for $1 \le i \le j < n$ and $s(\mathcal{T}_{+\infty}^{i}) > 0$ for $j < i \le n$;
	
	\item[\textbf{Case 4}] $s(\mathcal{T}_{+\infty}^{i}) < 0 < s(\mathcal{T}_{0}^{i})$ for all $1 \le i \le  n$;
	
	\item[\textbf{Case 5}] $s(\mathcal{T}_{+\infty}^{i}) < 0 < s(\mathcal{T}_{0}^{i})$ for all $1 \le i \le j < n$ and $s(\mathcal{T}_{0}^{i}) < 0$ for all $j < i \le n$;
	
	\item[\textbf{Case 6}] $s(\mathcal{T}_{+\infty}^{i}) < 0 < s(\mathcal{T}_{0}^{i})$ for all $1 \le i \le j < n$ and $s(\mathcal{T}_{+\infty}^{i}) > 0$ for all $j < i \le n$.
\end{itemize}
When $n \ge 3$, we have another case:
\begin{itemize}
	\item[\textbf{Case 7}] $s(\mathcal{T}_{+\infty}^{i}) < 0 < s(\mathcal{T}_{0}^{i})$ for all $1 \le i \le j < n-1$, $s(\mathcal{T}_{0}^{i}) < 0$ for all $j < i \le k < n$ and $s(\mathcal{T}_{+\infty}^{i}) > 0$ for all $k < i \le n$.
\end{itemize}

It follows from Proposition \ref{pr4.2} that in \textbf{Cases 4}--\textbf{7} there exists the unique  $\lambda_{i}^{\ast}$ such that $s(\tilde{\mathcal{T}}_{\lambda^{\ast}_{i}}^{i})=0$ for all $1 \le i \le j$. When $j \ge 2$, by rearrangement, we obtain
\[
\lambda_1^{\ast} \ge \lambda_2^{\ast} \ge \cdots \ge \lambda_j^{\ast}.
\]
For convenience, we introduce
\[
e_1 := (1, 0, \dots, 0), \quad e_2 := (0, 1, \dots, 0), \quad \dots, \quad e_n := (0, 0, \dots, 1).
\]
Based on Proposition \ref{pr4.2} and by the same discussion as in \cite[Section 4]{MR4612706}, one can obtain the following conclusions.
\begin{proposition}[{\cite[Theorem 4.7]{MR4612706}}]\label{pr4.8}
Assume that   {\rm(\textbf{A1})} and {\rm(\textbf{A2})} hold. For $\epsilon=0$, the solution $u$ of \eqref{eq4.3}   satisfies the following statements: 
	\begin{itemize}
		\item[\rm(i)] In {\rm\textbf{Case 1}}, there holds
		\[
		\lim_{t \to \infty} u(x,t) = (0, 0, \dots, 0) \quad \text{uniformly in } \overline{\Omega};
		\]
		\item[\rm(ii)] In {\rm\textbf{Case 2}}, there holds
		\[
		\lim_{t \to \infty} u(x,t) = (e_1 + e_2 + \dots + e_n) \infty \quad \text{uniformly in } \overline{\Omega};
		\]
		\item[\rm(iii)] In {\rm\textbf{Case 3}},  {\rm\textbf{Case 6}}, or {\rm\textbf{Case 7}}, let $l = j$ in {\rm\textbf{Case 3}} or {\rm\textbf{Case 6}} and $l = k$ in {\rm\textbf{Case 7}}. Then
		\[
		\lim_{t \to \infty} u(x,t) = (e_{l+1} + e_{l+2} + \dots + e_n) \infty \quad \text{uniformly in } \overline{\Omega};
		\]
		\item[\rm(iv)] In {\rm\textbf{Case 4}} or {\rm\textbf{Case 5}}, if $\lambda_1^{\ast} > \lambda_2^{\ast} > \dots > \lambda_j^{\ast} $, and $\kappa(x)$ is some constant functions, then there holds
		\[
		\lim_{t \to \infty} u(\cdot, t) = \theta_1 e_1 \quad \text{uniformly in } \overline{\Omega}.
		\]
	\end{itemize}
\end{proposition}
Consider the steady state equation of \eqref{eq4.3} with gene mutations
\begin{equation}\label{eq4.6}
	\begin{cases}
		\begin{aligned}
			0 &= 2(1+\epsilon m_{ii})\beta(s)\int_{\Omega} p_i(x,y)e^{-\mu_i(y)}u_i(y)\, \mathrm{d}y \\
			&\quad + 2\epsilon\sum_{j\neq i} m_{ij}\beta(s)\int_{\Omega} p_j(x,y)e^{-\mu_j(y)}u_j(y)\,\mathrm{d}y \\
			&\quad - \beta(s)u_i(x) - \kappa(x)u_i(x),  
		\end{aligned} & x\in\overline{\Omega},\ 1\leq i\leq n,\\[3em]
	\displaystyle	s = \sum\limits_{i=1}^n \int_{\Omega} \zeta(x) u_i(x)\,\mathrm{d}x.
	\end{cases}
\end{equation}
Similarly, using Proposition \ref{pr4.3}(ii) and the arguments in \cite[Theorem 5.1]{MR4612706}, we can remove Condition \eqref{eq4.5} and still obtain the same results as follows.
\begin{proposition}\label{pr4.4}
	Assume that   {\rm(\textbf{A1})}--{\rm(\textbf{A3})} hold. The following conclusions hold:
	\begin{itemize}
		\item [\rm(i)] If $s( \mathcal{T}_{\epsilon,0})<0$, then $0$ is the only nonnegative solution of \eqref{eq4.6} in $X^{+}$;
		
		\item [\rm(ii)] If $s(\mathcal{T}_{\epsilon,+\infty}) < 0 < s(\mathcal{T}_{\epsilon,0})$, then \eqref{eq4.6}  has a unique positive solution in {\rm Int}$(X^{+})$;
		
		\item [\rm(iii)] If $ s(\mathcal{T}_{\epsilon,+\infty})  > 0$, then $0$ is the only nonnegative solution of \eqref{eq4.6} in $X^{+}$.
	\end{itemize}
\end{proposition}

Based on Theorem \ref{th6.1}, Propositions \ref{pr4.8} and \ref{pr4.4}, and the proof in \cite[Theorem 1.4 and Corollary 1.5]{MR4612706}, we can remove Condition \eqref{eq4.5} and obtain the same conclusions  as follows.

\begin{proposition}[{\cite[Theorem  1.4]{MR4612706}}]
	Assume that  {\rm(\textbf{A1})}--{\rm(\textbf{A3})} hold. For any fixed $0 < \epsilon \leq \epsilon_0$, the following conclusions hold:
	\begin{enumerate}
		
		\item [\rm(i)] If $s( \mathcal{T}_{\epsilon,0})<0$, then $(0, 0, \cdots, 0)$ is the only nonnegative steady state of \eqref{eq4.3} and is globally asymptotically stable;
		
		\item [\rm(ii)] If $s(\mathcal{T}_{\epsilon,+\infty}) < 0 < s(\mathcal{T}_{\epsilon,0})$, then \eqref{eq4.3} has a unique positive continuous steady state $\theta_\epsilon$, moreover, all solutions of the system converge to this steady state provided that $\kappa$ is a constant function;
		
		\item [\rm(iii)] If $s(\mathcal{T}_{\epsilon,+\infty}) > 0$, then $(0, 0, \cdots, 0)$ is the only nonnegative steady state of \eqref{eq4.3}, and
		\begin{equation}\label{eq4.8}
		\lim_{t \to \infty} (u_1(x, t), u_2(x, t), \cdots, u_n(x, t)) = (\infty, \infty, \cdots, \infty) \quad \text{uniformly in } \overline{\Omega}.
		\end{equation}
	\end{enumerate}
\end{proposition}

\begin{proposition}[{\cite[Corollary 1.5]{MR4612706}}]
	Assume that {\rm(\textbf{A1})}--{\rm(\textbf{A3})} hold. Then the following conclusions hold:
	\begin{enumerate}
		
		\item [\rm(i)] If $\gamma(0)<0$, then $(0, 0, \cdots, 0)$ is the only nonnegative steady state of \eqref{eq4.3} and is globally asymptotically stable for any $0 < \epsilon \ll 1$;

		\item[\rm(ii)] If $\gamma(+\infty) < 0 < \gamma(0)$, then for any $0 < \epsilon \ll 1$, \eqref{eq4.3} has a unique positive continuous steady state $\theta_\epsilon$, and in the case of $\kappa$ being constant, all solutions of the system converge to this steady state. If in addition $s^{\ast}_{1} > s^{\ast}_{2} > \cdots > s^{\ast}_{n}$, then
		\[
		\lim_{\epsilon \to 0^+} \theta_\epsilon = (\theta_1, 0, \dots, 0) \quad \text{uniformly in } \overline{\Omega},
		\]
		where $ s^{\ast}_{i}$ satisfies $s(\mathcal{T}^{i}_{ s^{\ast}_{i}})=0 $;
		
		\item [\rm(iii)] If $\gamma(+\infty) > 0$, then $(0, 0, \dots, 0)$ is the only nonnegative steady state of \eqref{eq4.3} for any $0 < \epsilon \ll 1$, and \eqref{eq4.8} holds.
	\end{enumerate}
\end{proposition}

Note that the proof of the limiting profile of $\theta_\epsilon$ in \cite[Corollary 1.5]{MR4612706} depends on \cite[Theorem 2.7]{MR4612706}, which requires Condition \eqref{eq4.5} to guarantee that $\mathcal{T}_{\epsilon,\lambda}$ admits a principal eigenvalue for all $\epsilon,\lambda>0$. However, for the proof we only need $\mathcal{T}_{\epsilon,\lambda}$ to admit a principal eigenvalue for $(\epsilon,\lambda)\in(0,\delta)\times(s_{1}^\ast-\delta,s_{1}^\ast+\delta)$ for some $\delta>0 $.

In fact, there exist $\delta_1,\delta_2>0$ such that  $\gamma(\lambda)>-\frac{1}{4}$ for $\lambda\in (s_{1}^\ast-\delta_{1},s_{1}^\ast+\delta_{1})$, and for all $\epsilon \in \left( 0, \delta_{2}\right) $ and $\lambda\in (s_1^\ast-\delta_{1},s_1^\ast+\delta_{1})$,
\[
s(\mathcal{T}_{\epsilon,\lambda})\ge s(\mathcal{T}_{\epsilon,s_{1}^{\ast}+\delta_{1}})\ge \gamma (s_{1}^{\ast}+\delta_{1})  - \frac{1}{4}\ge  -\frac{1}{2},
\] 
where the first inequality follows from the monotonicity of $s(\mathcal{T}_{\epsilon,\cdot}) $.

Notice that $\max\limits_{x \in \overline{\Omega}} \left\lbrace -1 - \alpha(s^{\ast}_{i}) \kappa_{i}(x)\right\rbrace \le -1$. Set $\delta= \min\left\lbrace \delta_{1}, \delta_{2} \right\rbrace $. Then, by Proposition \ref{pr2.1} and Theorem \ref{th6.1}, we obtain that $\mathcal{T}_{\epsilon,\lambda}$ admits a principal eigenvalue for $(\epsilon,\lambda)\in(0,\delta)\times(s_{1}^{\ast}-\delta,s_{1}^{\ast}+\delta)$. Consequently, the limiting profile of $\theta_{\epsilon}$ can be established by the same argument as in \cite{MR4612706}.

\subsection{The reducible case}
In this subsection, we  study the dynamics of system \eqref{eq4.3} when $M$ is reducible, focusing on the simple case  $n=2$. According to  Proposition \ref{pr3.1}, there exists a permutation matrix such that $M$ is a lower block-triangular matrix.  Therefore, without loss of generality, we may assume
\[
M= \left( \begin{array}{ll}
	m_{11} & 0\\
	m_{21} & 0
\end{array}\right),
\]
with $-m_{11}=m_{21}>0$ and $-m_{22}=m_{12}=0$. These equalities enforce probability conservation, which  guarantees that every daughter cell inherits a genetic type with certainty. In this case, we rewrite \eqref{eq4.3} as follows:
\begin{equation}\label{eq4.12}
	\begin{cases}
		\begin{aligned}
			\frac{\partial u_1(x,t)}{\partial t} = & 2(1+\epsilon m_{11})\beta(s(t))\int_{\Omega} p_1(x,y)e^{-\mu_1(y)}u_1(y,t)\, \mathrm{d}y \\
			&  - \beta(s(t))u_1(x,t) - \kappa(x)u_1(x,t), 
		\end{aligned} &  (x,t)\in\overline{\Omega}\times(0,\infty), \\[2em]
		\begin{aligned}
			\frac{\partial u_2(x,t)}{\partial t} = & 2 \beta(s)\int_{\Omega} p_2(x,y)e^{-\mu_2(y)}u_2(y)\, \mathrm{d}y - \beta(s)u_2(x) \\
			& - \kappa(x)u_2(x)+ 2\epsilon m_{21}\beta(s)\int_{\Omega} p_1(x,y)e^{-\mu_1(y)}u_1(y)\,\mathrm{d}y ,  
		\end{aligned} & (x,t)\in\overline{\Omega}\times(0,\infty),\\
		\displaystyle	s(t) = \sum\limits_{i=1}^n \int_{\Omega} \zeta(x) u_i(x,t) \, \mathrm{d}x, & t\in[0,\infty),\\[1em]
		\left(u_1(x,0),u_2(x,0) \right)  = \left( u_{1,0}(x),  u_{2,0}(x)\right) , & x\in\overline{\Omega}.
	\end{cases}
\end{equation}
Recall that    $\hat{\mathcal{T}}_{\epsilon,\lambda}^{i} : C(\overline{\Omega}) \to C(\overline{\Omega})$:
\[
\left[ \hat{\mathcal{T}}_{\epsilon,\lambda}^{i} \phi\right](x)  = \left( 1+ \epsilon m_{ii}\right) \left[ \mathcal{P}_i \phi \right](x) - (1+\frac{\kappa(x)}{\beta(\lambda)}) \phi(x), \quad x\in \overline{\Omega}, \ i =1,2.
\] 
To simplicity,  we replace $\frac{1}{\beta}$ by $\lambda$ in $\hat{\mathcal{T}}_{\epsilon,\lambda}^{i}$ and denote the resulting operator  by  $\tilde{\mathcal{T}}_{\epsilon,\lambda}^{i}$.
The following cases are considered:
\begin{itemize}
	\item[\textbf{Case I}] $ s(\hat{\mathcal{T}}_{\epsilon,0}^{i}) < 0$ for all $1 \le i \le 2$;
	
	\item[\textbf{Case II}] $s(\hat{\mathcal{T}}_{\epsilon,0}^{1}) < 0$  and  $s(\hat{\mathcal{T}}_{\epsilon,+\infty}^{2}) < 0 < s(\hat{\mathcal{T}}_{\epsilon,0}^{2})$;
	
	\item[\textbf{Case III}]  $s(\hat{\mathcal{T}}_{\epsilon,0}^{1}) < 0$  and  $s(\hat{\mathcal{T}}_{\epsilon,+\infty}^{2}) > 0$;
	
	\item[\textbf{Case IV}] $s(\hat{\mathcal{T}}_{\epsilon,+\infty}^{1}) > 0$ and $s(\hat{\mathcal{T}}_{\epsilon,0}^{2}) < 0$;
	
	\item[\textbf{Case V}] $s(\hat{\mathcal{T}}_{\epsilon,+\infty}^{1}) > 0$ and $s(\hat{\mathcal{T}}_{\epsilon,+\infty}^{2}) < 0 < s(\hat{\mathcal{T}}_{\epsilon,0}^{2})$;
	
	\item[\textbf{Case VI}] $s(\hat{\mathcal{T}}_{\epsilon,+\infty}^{i}) > 0$ for all $1 \le i \le 2$;
	
	\item[\textbf{Case VII}] $s(\hat{\mathcal{T}}_{\epsilon,+\infty}^{1}) < 0 < s(\hat{\mathcal{T}}_{\epsilon,0}^{1})$  and $s(\hat{\mathcal{T}}_{\epsilon,0}^{2}) < 0$;
	
	\item[\textbf{Case VIII}] $s(\hat{\mathcal{T}}_{\epsilon,+\infty}^{i}) < 0 < s(\hat{\mathcal{T}}_{\epsilon,0}^{i})$ for all $1 \le i \le 2$;
	
	\item[\textbf{Case IX}] $s(\hat{\mathcal{T}}_{\epsilon,+\infty}^{1}) < 0 < s(\hat{\mathcal{T}}_{\epsilon,0}^{1})$  and $s(\hat{\mathcal{T}}_{\epsilon,+\infty}^{2}) > 0$.

\end{itemize}

Here we cannot rearrange  the system components  again to simplify the above cases, because assuming that $M$ is lower triangular already implies one such reordering.
Consider the steady state equation of \eqref{eq4.12}
\begin{equation}\label{eq4.9}
	\begin{cases}
	\displaystyle	0 = 2(1+\epsilon m_{11})\beta(s)\int_{\Omega} p_1(x,y)e^{-\mu_1(y)}u_1(y)\, \mathrm{d}y - \beta(s)u_1(x) - \kappa(x)u_1(x),  
		& x\in\overline{\Omega},\\[1em]
		\begin{aligned}
			0 &= 2 \beta(s)\int_{\Omega} p_2(x,y)e^{-\mu_2(y)}u_2(y)\, \mathrm{d}y - \beta(s)u_2(x) - \kappa(x)u_2(x)\\
			&\quad + 2\epsilon m_{21}\beta(s)\int_{\Omega} p_1(x,y)e^{-\mu_1(y)}u_1(y)\,\mathrm{d}y ,  
		\end{aligned} & x\in\overline{\Omega},\\[2em]
		\displaystyle	
		s =  \int_{\Omega} \zeta(x)\left(  u_1(x) + u_2(x)\right) \,\mathrm{d}x.
	\end{cases}
\end{equation}

\begin{lemma}
	Assume that  {\rm(\textbf{A1})} and {\rm(\textbf{A2})} hold. 
	Then the following statements hold for \eqref{eq4.9}:
	\begin{itemize}
		\item [\rm(i)] In {\rm\textbf{Case I}},  {\rm\textbf{Case III}},  {\rm\textbf{Case IV}},  {\rm\textbf{Case VI}}, or {\rm\textbf{Case IX}}, $(0,0)$ is the only nonnegative solution  in $X^{+}$; 
		
		\item [\rm(ii)] In {\rm\textbf{Case II}} or {\rm\textbf{Case V}}, \eqref{eq4.9} has a unique solution $(0,\phi_{2})$ in $X^{+}\setminus\left\lbrace 0\right\rbrace $;
		
		\item [\rm(iii)] In {\rm\textbf{Case VII}}, \eqref{eq4.9} has a unique positive  solution $(\phi_{1},\phi_{2})$ in $X^{+}\setminus\left\lbrace 0\right\rbrace $, moreover, $(\phi_{1},\phi_{2})\in {\rm Int}(X^{+})$;
		
		\item[\rm(iv)] In {\rm\textbf{Case VIII}}, we have: 
		\begin{itemize}
			\item[\rm(a)] If $\lambda_{\epsilon,1}^{\ast}< \lambda_{\epsilon,2}^{\ast}$, then \eqref{eq4.9} has a unique solution $(0,\phi_{2})$ in $X^{+}\setminus\left\lbrace 0\right\rbrace $;
			
			\item[\rm(b)] If $\lambda_{\epsilon,1}^{\ast}> \lambda_{\epsilon,2}^{\ast}$, then \eqref{eq4.9} has  exactly two solutions $(0,\phi_{2})$ and $(\phi_{1},\phi_{2})$ in $X^{+}\setminus\left\lbrace 0\right\rbrace $;
		\end{itemize}
		where $ \lambda_{\epsilon,i}^{\ast}$ is the unique solution of $s(\tilde{\mathcal{T}}_{\epsilon,\lambda}^{i})=0$ for $i=1,2$. 
	\end{itemize}

\end{lemma}

\begin{proof}
	
Let $(\phi_{1},\phi_{2})$ be the nonnegative solution for  \eqref{eq4.9}.
According to Proposition \ref{pr4.2}, in \textbf{Cases I}--\textbf{VI}, we obtain that $\phi_{1}\equiv0$. Applying Proposition \ref{pr4.2} again yields the conclusions for these cases.

In \textbf{Case VII}, by Propositions \ref{pr2.1}, \ref{pr3.2} and \ref{pr4.3}, and Theorem \ref{th6.1},  there exists a unique $s$ such that $s(\mathcal{T}_{\epsilon,s})=s(\hat{\mathcal{T}}_{\epsilon,s}^{1}) =0$ and $\mathcal{T}_{\epsilon,s}$ admits a principal eigenvalue with an associated eigenfunction $(\varphi_1, \varphi_2)\in X^{+}\setminus\left\lbrace 0\right\rbrace$.  If $\varphi_{1}\equiv0$, then  $\varphi_{2}\equiv0 $ by Proposition \ref{pr4.2}, a contradiction. Hence, $\varphi_{1}\not\equiv0$ and  it is easy to show $\varphi_{1}(x)>0 $ for all $x\in \overline\Omega$. Notice that 
\begin{equation}\label{eq4.10}
\begin{aligned}
	0 &= 2 \beta(s)\int_{\Omega} p_2(x,y)e^{-\mu_2(y)}\varphi_2(y)\, \mathrm{d}y - \beta(s)\varphi_2(x) - \kappa(x)\varphi_2(x)\\
	&\quad + 2\epsilon m_{21}\beta(s)\int_{\Omega} p_1(x,y)e^{-\mu_1(y)}\varphi_1(y)\,\mathrm{d}y.
\end{aligned} 
\end{equation}
It follows that $\varphi_{2}(x)>0$ for all  $x\in \overline\Omega$.  Moreover, $\varphi_{1} $ is a principal eigenfunction  of $\hat{\mathcal{T}}_{\epsilon,s}^{1}$, and $\varphi_{2}= -\left( \hat{\mathcal{T}}_{\epsilon,s}^{2}\right)^{-1}\psi $ with $\psi(x)=2\epsilon m_{21}\int_{\Omega} p_1(x,y)e^{-\mu_1(y)}\varphi_1(y)\,\mathrm{d}y$, which implies $(\varphi_1, \varphi_2)$ is unique up to a scalar multiple. Clearly, 
\[
\phi_{i}= \dfrac{s\varphi_{i}}{\int_{\Omega}\zeta(x)\left( \varphi_{1}(x)+\varphi_{2}(x)\right)\, \mathrm{d}x }, \quad i =1,2.
\]

In \textbf{Case VIII}(a),  we first show that $\phi_{1}\equiv0$. Suppose otherwise. Then $\phi_{1}(x)>0 $ for all $x\in \overline{\Omega}$ and $s(\hat{\mathcal{T}}_{\epsilon,s}^{1})=0$. This follows that  $s= \beta^{-1}(\frac{1}{\lambda_{\epsilon,1}^{\ast}})$ and  \eqref{eq4.10} holds for such $s$. Then $s(\tilde{\mathcal{T}}_{\epsilon,\lambda_{\epsilon,1}^{\ast}}^{2})\le 0$, which implies $s(\tilde{\mathcal{T}}_{\epsilon,\lambda_{\epsilon,2}^{\ast}}^{2})< 0$, a contradiction. By  Proposition \ref{pr4.2}, there exists a unique solution  $(0, \phi_{2})$.

In \textbf{Case VIII}(b), set $s= \beta^{-1}(\frac{1}{\lambda_{\epsilon,1}^{\ast}})$. As in \textbf{Case VII}, we obtain that \eqref{eq4.9} has a positive solution $(\phi_{1},\phi_{2})$.   By  Proposition \ref{pr4.2}, we also know the existence of the solution $(0, \phi_{2})$.

In \textbf{Case IX},   if $\phi_{1}\equiv0$, then Proposition \ref{pr4.2} yields $\phi_{2}\equiv0 $. If $\phi_{1}\not\equiv0$, it is easy to see $\phi_{1}(x)>0 $ for all $x\in \overline\Omega$. As in \textbf{Case VIII}(a), there exists some $s$ such that 
$s(\hat{\mathcal{T}}_{\epsilon,s}^{2})\le 0$, which implies $s(\hat{\mathcal{T}}_{\epsilon,+\infty}^{2})\le 0$, a contradiction. Hence, $(\phi_{1},\phi_{2})= (0,0)$. 
\end{proof}

\begin{theorem}\label{th5.1}
	Assume that  {\rm(\textbf{A1})} and {\rm(\textbf{A2})} hold, and $\kappa$ is a constant function. Let $u(x,t):=(u_{1}(x,t),u_{2}(x,t))$ be the solution for  \eqref{eq4.12} with $u_{1,0}, u_{2,0}\ge , \not\equiv  0$, then
	 following statements hold:
	
	\begin{itemize}
		\item [\rm(i)] In {\rm\textbf{Case I}}, there holds
		\[
		\lim\limits_{t\to +\infty} (u_{1}(x,t),u_{2}(x,t))=(0,0) \quad \text{ uniformly in } \overline{\Omega};
		\]
		
		\item[\rm(ii)] In {\rm\textbf{Case III}} or {\rm\textbf{Case IX}}, there holds
		\[
		\lim\limits_{t\to +\infty} (u_{1}(x,t),u_{2}(x,t))=(0,+\infty) \quad \text{ uniformly in } \overline{\Omega};
		\]
		
		\item[\rm(iii)] In {\rm\textbf{Case IV}}, {\rm\textbf{Case V}} or {\rm\textbf{Case VI}}, there holds
		\[
		\lim\limits_{t\to +\infty} (u_{1}(x,t),u_{2}(x,t))=(+\infty,+\infty) \quad \text{ uniformly in } \overline{\Omega};
		\]

		\item[\rm(iv)] In {\rm\textbf{Case II}} or {\rm\textbf{Case VIII}(a)}, there holds
			\[
			\lim\limits_{t\to +\infty} (u_{1}(x,t),u_{2}(x,t))=(0,\phi_{2}(x)) \quad \text{ uniformly in } \overline{\Omega};
			\]
	
	   \item [\rm(v)] In {\rm\textbf{Case VII}} or {\rm\textbf{Case VIII}(b)}, there holds
	   \[
	   \lim\limits_{t\to +\infty} (u_{1}(x,t),u_{2}(x,t))=(\phi_{1}(x),\phi_{2}(x)) \quad \text{ uniformly in } \overline{\Omega}.
	   \]

	\end{itemize}

\end{theorem}

\begin{proof}
	Set
	\[
	\tau = \int_0^t \beta(s(z))\,\mathrm{d}z, \qquad U_i(x,\tau)=u_i(x,t),\quad i=1,2,
	\]
	and 
	\[
	S(\tau)=S_1(\tau)+S_2(\tau),\qquad S_i(\tau)=\int_\Omega \zeta(x)U_i(x,\tau)\,\mathrm{d}x.
	\]
Then 
\begin{equation}\label{eq4.13}
\begin{cases}
	\displaystyle \frac{\partial U_1}{\partial\tau}=2(1+\epsilon m_{11})\int_\Omega p_1(x,y) e^{-\mu_1(y)}U_1(y, \tau)\,\mathrm{d}y
	 -U_1-\frac{\kappa}{\beta(S)}U_1,  \\[1em]
	 \begin{aligned}
	 	\frac{\partial U_2}{\partial\tau}= & 2\int_\Omega p_2(x,y) e^{-\mu_2(y)}U_2(y, \tau)\,\mathrm{d}y-U_2-\frac{\kappa}{\beta(S)}U_2\\
	 	&+2\epsilon m_{21}\int_\Omega p_1(x,y) e^{-\mu_1(y)}U_1(y, \tau)\,\mathrm{d}y. 
	 \end{aligned}
\end{cases}
\end{equation}

(i) It is easy to see that $U :=(U_{1},U_{2})$ is a subsolution of \eqref{eq4.13} with $\beta(S)$ replaced  by $\beta(0)$. Since $s(\mathcal{T}_{\epsilon,0})=\max\limits_{i=1,2} s(\hat{\mathcal{T}}_{\epsilon,0}^{i}) < 0$, there exist $M>0$ and $\omega>0$ such that 
\[
\left\| U(\tau,\cdot)\right\|\le M e^{-\omega \tau}\left\| U(0,\cdot)\right\| \quad \text{ for all } \tau >0.
\]
Consequently,  $\lim\limits_{\tau \to +\infty}\left\| U(\tau,\cdot)\right\|=0$.  As $t\to +\infty$, we have $\tau\to +\infty$. Therefore,
\[
\lim\limits_{t\to +\infty} (u_{1}(x,t),u_{2}(x,t))=(0,0) \quad \text{ uniformly in } \overline{\Omega}.
\]

(ii) Notice that 
\[
\frac{\partial U_2}{\partial\tau}\ge  2\int_\Omega p_2(x,y) e^{-\mu_2(y)}U_2(y, \tau)\,\mathrm{d}y-U_2-\frac{\kappa}{\beta(+\infty)}U_2. 
\]
Since $s(\hat{\mathcal{T}}_{\epsilon,+\infty}^{2}) > 0$, we have $\lim\limits_{\tau \to +\infty} U_{2}(\tau ,x)= +\infty $ uniformly in $\overline{\Omega}$.  Then there exist  $T$ and $M$ such that  $S(\tau)\ge M$ for  all $\tau \ge T$  and  $s(\hat{\mathcal{T}}_{\epsilon,M}^{1}) <0$. Note that 
\[
\frac{\partial U_1}{\partial\tau}\le 2(1+\epsilon m_{11})\int_\Omega p_1(x,y) e^{-\mu_1(y)}U_1(y, \tau)\,\mathrm{d}y
-U_1-\frac{\kappa}{\beta(M)}U_1.
\]
It follows that $\lim\limits_{\tau \to +\infty} U_{1}(\tau ,x)= 0 $ uniformly in $\overline{\Omega}$. 
As $t\to +\infty$, we have $\tau\to +\infty$. Therefore,
\[
\lim\limits_{t\to +\infty} (u_{1}(x,t),u_{2}(x,t))=(0,+\infty) \quad \text{ uniformly in } \overline{\Omega}.
\]

(iii) Since $s(\hat{\mathcal{T}}_{\epsilon,+\infty}^{1}) > 0$, and 
\[
\frac{\partial U_1}{\partial\tau}\ge 2(1+\epsilon m_{11})\int_\Omega p_1(x,y) e^{-\mu_1(y)}U_1(y, \tau)\,\mathrm{d}y
-U_1-\frac{\kappa}{\beta(+\infty)}U_1,
\]
we have $\lim\limits_{\tau \to +\infty} U_{1}(\tau ,x)= +\infty $ uniformly in $\overline{\Omega}$. 
Due to 
\[
\frac{\partial U_2}{\partial\tau}\ge  2\int_\Omega p_2(x,y) e^{-\mu_2(y)}U_2(y, \tau)\,\mathrm{d}y-U_2-\frac{\kappa}{\beta(+\infty)}U_2 +  2\epsilon m_{21}\int_\Omega p_1(x,y) e^{-\mu_1(y)}U_1(y, \tau)\,\mathrm{d}y,
\]
we have 
\[
U_{2}(x,\tau)\ge e^{-(1+\frac{\kappa}{\beta(+\infty)})\tau}U_{2,0}(x)+2\epsilon m_{21}\int_{0}^{\tau}   e^{-(1+\frac{\kappa}{\beta(+\infty)})(\tau -r)}\int_\Omega p_1(x,y) e^{-\mu_1(y)}U_1(y, r)\,\mathrm{d}y\mathrm{d}r.
\]
Since $\lim\limits_{\tau \to +\infty}\int_{0}^{\tau}   e^{-(1+\frac{\kappa}{\beta(+\infty)})(\tau -r)}\int_\Omega p_1(x,y) e^{-\mu_1(y)}U_1(y, r)\,\mathrm{d}y\mathrm{d}r = +\infty$ uniformly in $\overline{\Omega}$, it follows that $\lim\limits_{\tau \to +\infty} U_{2}(\tau ,x)= +\infty $ uniformly in $\overline{\Omega}$.
As $t\to +\infty$, we have $\tau\to +\infty$. Therefore,
\[
\lim\limits_{t\to +\infty} (u_{1}(x,t),u_{2}(x,t))=(+\infty,+\infty) \quad \text{ uniformly in } \overline{\Omega};
\]

(iv)
Set  $V_i(x,\tau)=\frac{U_i(x,\tau) }{S (\tau)}$  and 
\[
  c(\tau)=2\int_{\Omega} \int_{\Omega} \zeta(x)p_1(x,y) e^{-\mu_1(y)}V_1(y, \tau)\,\mathrm{d}y \mathrm{d}x  + 2 \int_{\Omega} \int_\Omega \zeta(x)p_2(x,y) e^{-\mu_2(y)}V_2(y, \tau)\,\mathrm{d}y \mathrm{d}x 
\]
 By calculation, we have 
\begin{equation}\label{eq4.14}
	\begin{aligned}
		\frac{\mathrm{d}S}{\mathrm{d}\tau} =  \left(  c(\tau)  
		- 1 - \frac{\kappa}{\beta(S)} \right) S,
	\end{aligned}
\end{equation}
and 
\[
\begin{cases}
\displaystyle	\frac{\partial V_1}{\partial \tau} = 2(1+\epsilon m_{11}) \int_\Omega p_1(x,y) e^{-\mu_1(y)}V_1(y, \tau)\,\mathrm{d}y -  V_1c(\tau),
	\\[1em]
\displaystyle	\frac{\partial V_2}{\partial \tau} =  2 \int_\Omega p_2(x,y) e^{-\mu_2(y)}V_2(y, \tau)\,\mathrm{d}y 
	+ 2\epsilon m_{21} \int_\Omega p_1(x,y) e^{-\mu_1(y)}V_1(y, \tau)\,\mathrm{d}y-  V_2c(\tau).
\end{cases}
\]
Set 
\[
\lambda= \int_{\Omega}\zeta(x)\phi_{2}(x)\, \mathrm{d}x, \quad W_{i}(x,\tau)= e^{\int_{0}^{\tau }c(\sigma)\, \mathrm{d}\sigma-(1+\frac{\kappa}{\beta(\lambda)})\tau} V_{i}(x,\tau), \ i=1,2.
\]
Then, 
\begin{equation}
	\dfrac{\partial W}{\partial \tau}= \mathcal{T}_{\epsilon,\lambda}W,
\end{equation}
where $W=(W_{1},W_{2})$. 
Denote by $\Phi(\tau)$ the analytic semigroup on $X$ generated by $\mathcal{T}_{\epsilon,\lambda}$. 
Recall that $s(\mathcal{T}_{\epsilon,\lambda})=\max\limits_{i=1,2} s(\hat{\mathcal{T}}_{\epsilon,\lambda}^{i})$. 
Clearly, $s(\hat{\mathcal{T}}_{\epsilon,\lambda}^{1})<0$ and $s(\hat{\mathcal{T}}_{\epsilon,\lambda}^{2})=0$. 
Since $-(1+\frac{\kappa}{\beta(\lambda)})<0$, it follows that $s(\mathcal{T}_{\epsilon,\lambda})=0$ is an isolated  principal eigenvalue for $\mathcal{T}_{\epsilon,\lambda}$. 
Note that for any $\phi=(\phi_{1},\phi_{2})$ satisfying  $ \mathcal{T}_{\epsilon,\lambda}\phi=0$, we have $\phi_{1}\equiv0$ and $\hat{\mathcal{T}}_{\epsilon,\lambda}^{2}\phi_{2}=0$. According to Theorem \ref{th2.1}, we obtain that  $s(\mathcal{T}_{\epsilon,\lambda})=0$ is an isolated algebraically simple principal eigenvalue for $\mathcal{T}_{\epsilon,\lambda}$ with a principal eigenfunction $\phi=(0, \phi_{2})$.
Let $ P $ denote the corresponding spectral projection, $ X_1 := PX $ and $ X_2 := (I - P)X $. Clearly, we have
\[
X_1 = \ker  {\mathcal{T}}_{\epsilon,\lambda }  = \{c\phi \mid c \in \mathbb{C}\}, \quad 
\Phi(\tau)\psi= \psi \text{ for  all } \psi \in X_{1}, \tau \ge 0.
\]
It follows from \cite[Theorem 1.5.2]{MR610244} that $ X = X_1 \oplus X_2 $ and for   $ i = 1, 2 $, $ X_i $ is invariant under $  {\mathcal{T}}_{\epsilon,\lambda }  $.
Moreover, there exist $C > 0$ and $\eta > 0$ such that
\[
\|\Phi(\tau)|_{X_{2}}\|_X \le Ce^{-\eta\tau}
\]
for all $\tau > 0$, see \cite[ Theorem 1.5.3]{MR610244}. Rewrite 
\[
W_{0}=c_{0}\phi + \hat W_{0},
\]
where $\hat W_{0}\in X_{2}$. Then $\lim\limits_{\tau \to +\infty}\left\| W(\cdot, \tau )-  c_{0}\phi\right\|=0 $. 
Now we show that $c_{0}>0$. 
Obviously, $P\psi = \lim\limits_{\tau \to +\infty}\Phi(\tau)\psi$ for any $\psi \in X$. Then $P$ is positive. Define a function $f$ on $X$ by 
\[
f(\psi)=c, \text{ where } c \text{ satisfies } P\psi=c\psi.
\]
We claim that $f(\psi)>0$ for all $\psi\in$ Int$(X^{+})$. If not, there exists a $\psi\in$ Int$(X^{+})$ such that $f(\psi)=0$. Hence, there exists $\delta>0$ such that $\psi\pm\delta\eta\in X^{+}$ for all $\eta\in X$ with $\left\| \eta\right\|=1 $. Then $\pm f(\eta)\ge 0$, which means $f\equiv0$, a contradiction.

If $c_{0}=0$, then $W(\cdot, \tau )\to 0$ as $\tau \to +\infty$. Since $u_{1,0}, u_{2,0}\ge , \not\equiv 0$, then $u(\cdot,t)\in$ Int$(X^{+})$  for any $t>0$, and similarly $W(\cdot,\tau)\in$ Int$(X^{+})$. Then for a $\tau_{0}>0$, we can write $W(\cdot, \tau_{0})= c'\phi+ \hat W(\cdot, \tau_{0})$, where $\hat W(\cdot, \tau_{0})\in X_{2}$ and $c'>0$. It follows that $W(\cdot, \tau )\to c'\phi$ as $\tau \to +\infty$, a contradiction. Therefore, $c_{0}>0$. 

Due to 
\[
V_{i}(x,\tau)= \dfrac{W_{i}(x,\tau)}{\int_{\Omega}\zeta(x)\left(W_{1}(x,\tau) +W_{2}(x,\tau)\right)\, \mathrm{d}x }, 
\]
we have 
\[
\lim\limits_{\tau \to +\infty} V(\cdot, \tau )= \frac{\phi}{\int_{\Omega}\zeta(x)\phi_{2}(x ) \, \mathrm{d}x }, 
\]
and 
\[
\lim\limits_{\tau \to+\infty} c(\tau )= \frac{2 \int_{\Omega} \int_\Omega \zeta(x)p_2(x,y) e^{-\mu_2(y)}\phi_2(y, \tau)\,\mathrm{d}y \mathrm{d}x }{\int_{\Omega}\zeta(x)\phi_{2}(x ) \, \mathrm{d}x }. 
\]
Consider the limiting system for \eqref{eq4.14} as follows:
\begin{equation}\label{eq4.15}
 \frac{\mathrm{d}S}{\mathrm{d}\tau} =  \left(  \frac{2 \int_{\Omega} \int_\Omega \zeta(x)p_2(x,y) e^{-\mu_2(y)}\phi_2(y, \tau)\,\mathrm{d}y \mathrm{d}x }{\int_{\Omega}\zeta(x)\phi_{2}(x ) \, \mathrm{d}x }
 - 1 - \frac{\kappa}{\beta(S)} \right) S.  
\end{equation}
It is easy to verify  that system  \eqref{eq4.15} admits a unique positive steady state 
\[
S_{0}=\int_{\Omega}\zeta(x)\phi_{2}(x ) \, \mathrm{d}x,
\]
and it is asymptotically stable on $S>0$. By the theory of asymptotically autonomous semiflows (see \cite{MR1175102}) or the theory of chain
transitive sets (see \cite[Section 1.2.1]{MR3643081}), or by a similar discussion as in \cite[Theorem 3.5]{MR4612706}
we have 
\[
\lim\limits_{\tau\to +\infty}S(\tau)= \int_{\Omega}\zeta(x)\phi_{2}(x ) \, \mathrm{d}x.
\]
Therefore, we obtain that
\[
	\lim\limits_{t\to +\infty} (u_{1}(x,t),u_{2}(x,t))=(0,\phi_{2}(x)) \quad \text{ uniformly in } \overline{\Omega}.
\]

(v) By a similar argument of (iv) where $\lambda= \int_{\Omega}\zeta(x)\left( \phi_{1}(x)+\phi_{2}(x)\right) \, \mathrm{d}x$ and $\phi=(\phi_{1}, \phi_{2})$, we can obtain the desired conclusion.
\end{proof}

From the above conclusions, we know that $s(\mathcal{T}_{\lambda}^{i})$ (resp. $s(\hat{\mathcal{T}}_{\epsilon,\lambda}^{i}) $) plays an important role in the dynamical behavior of the system \eqref{eq4.3}. In \cite{MR4612706}, Su et al. provide some criteria for directly comparing $s(\mathcal{T}_{0}^{i})$ and $s(\mathcal{T}_{+\infty}^{i})$ with $0$ for the case $p(x,y) = p_1(x)p_2(y)$. In view of the definition of $\mathcal{T}_{\lambda}^{i}$, we can first compute the unique  $\lambda^{\ast}_{i}$ such that $s(\tilde{\mathcal{T}}_{\lambda^{\ast}_{i}}^{i}) = 0$, then compare $\lambda^{\ast}_{i}$ with the range of $\alpha$, and thereby obtain the dynamical behavior of the system.

Next, we focus on how to numerically compute the unique  $\lambda^{\ast}_{i}$ such that $s(\tilde{\mathcal{T}}_{\lambda^{\ast}_{i}}^{i}) = 0$ in the general case. According to Proposition \ref{pr4.3}, there exists a unique pair $(\lambda^{\ast}_{i}, \phi_{i})$ such that 
\[
\left[ \mathcal{P}_i \phi_{i}\right] (x) -\phi_{i}(x)- \lambda_{i}^{\ast} \kappa_{i}(x)\phi_{i}(x)=0, \quad x\in \overline{\Omega}. 
\]
Define $\Gamma_{\lambda}:  Y \to Y$ as 
\[
\left[ \Gamma_{\lambda}\phi\right] (x) = \dfrac{\left[\mathcal{P}_i \phi_{i}\right] (x)}{1+\lambda\kappa_{i}(x)}.
\]
It is easy to verify that  $\Gamma_{\lambda}$ is a compact operator on $Y$ for each $\lambda>0$, and that $\Gamma_{\lambda}$ is strictly decreasing in $\lambda$ for $\lambda > 0$. Moreover,  $\lambda^{\ast}_{i}$ is the unique positive constant satisfying  $r(\Gamma_{\lambda^{\ast}_{i}})=1$.

Choose $\phi_{0}\in$ Int$(Y^{+})$ and define 
\[
a_{n}=\left\| \Gamma_{\lambda}\phi_{n-1}\right\|, \quad \phi_{n}= \frac{\Gamma_{\lambda}\phi_{n-1}}{a_{n}}.
\]
According to  \cite[Lemma 2.5]{MR3992071}, if $\lim\limits_{n \to +\infty} a_{n}$ exists, then $r(\Gamma_{\lambda})=\lim\limits_{n \to +\infty} a_{n}$. Next, we show that $\lim\limits_{n \to +\infty} a_{n}$ exists and is independent of $\phi_{0}$.  For convenience, let $\rho = r(\Gamma_{\lambda})$ and, by the Krein--Rutman theorem, let $u$ be a strictly positive function satisfying $\Gamma_{\lambda} u = \rho u$.

\begin{proposition}
	Assume that {\rm (\textbf{A1})} holds. For any $\phi_{0}\in \mathrm{Int}(Y^{+})$,   $\lim\limits_{n\to\infty} a_n = \rho$.
\end{proposition}

\begin{proof}
	
	Let $P$ denote the projections associated with $\left\lbrace \rho\right\rbrace $, and set $Y_{1}:= PY$ and $Y_{2}:=(I-P)Y$. Then, we have 
	\[
	P^2=P, \quad \Gamma_\lambda P = P\Gamma_\lambda = \rho P.
	\]
	By \cite[Theorem 1.5.2]{MR610244},  $ Y = Y_1 \oplus Y_2 $ and each $ Y_i $ is invariant under $  \Gamma_\lambda  $;  moreover, $\sigma(\Gamma_{\lambda}\mid_{Y_{2}})= \sigma(\Gamma_{\lambda})\setminus\left\lbrace \rho \right\rbrace $.

	Define $N = \Gamma_\lambda - \rho P$. It follows that  $PN = NP = 0$ and $\sigma(N)=\left( \sigma(\Gamma_{\lambda})\setminus\left\lbrace \rho \right\rbrace \right) \cup\left\lbrace0 \right\rbrace $,   so in particular  $r(N) < \rho$. 
	From $PN=NP=0$, we obtain, for every $n\in\mathbb{N}$:
	\[
	\Gamma_\lambda^n = (\rho P + N)^n = \rho^n P + N^n.
	\]

	Now fix any $\eta$ such that $r(N) < \eta < \rho$. By Gelfand's formula, $r(N)=\lim\limits_{n\to\infty}\left\| N^n\right\|^{\frac{1}{n}}$. Hence, 
    there exists $n_0\in\mathbb{N}$ such that $\left\| N^n\right\|^{\frac{1}{n}}\le \eta$ for all $n\ge n_0$, i.e. $\|N^n\|\le \eta^n$ for $n\ge n_0$. 
	For $1\le n<n_0$, set $C_1 = \max\limits_{1\le n<n_0}\frac{\left\| N^n\right\| }{\eta^n}$ and let $C = \max\{C_1,1\}$; then $\left\| N^n\right\| \le C\eta^n$ holds for every $n\ge 1$.
	
	By the same argument as in   Theorem \ref{th5.1}(iv),  for any $\phi_{0}\in \mathrm{Int}(Y^{+})$, there exists $c>0$ such that  $P\phi_0 = c u$. 
	Applying $\Gamma_\lambda^n$ to $\phi_0$ gives
	\[
	\Gamma_\lambda^n \phi_0 = \rho^n P\phi_0 + N^n\phi_0 = c\rho^n u + N^n\phi_0.
	\]
	From the definition of the iteration we have $\Gamma_\lambda^n\phi_0 = \left( \prod_{k=1}^n a_k\right) \phi_n$ and $\left\| \phi_n\right\| =1$. Thus $\prod_{k=1}^n a_k = \left\| \Gamma_\lambda^n\phi_0\right\| $. 
   Consequently,
	\[
	a_n = \frac{\left\| \Gamma_\lambda^n\phi_0\right\| }{\left\| \Gamma_\lambda^{n-1}\phi_0\right\| }=\frac{\left\| c\rho^n u + N^n\phi_0 \right\| }{\left\| c\rho^{n-1} u + N^{n-1}\phi_0\right\| }
	= \rho \frac{\left\| c u + \frac{1}{\rho^{n}}N^n\phi_0 \right\|}{\left\| c  u +  \frac{1}{\rho^{n-1}}N^{n-1}\phi_0\right\|}.
	\]
	Since $\eta<\rho$, we have $\frac{1}{\rho^{n}}N^n\phi_0 \to 0$ as $n\to\infty$. 
	Therefore,
	\[
	\lim_{n\to\infty} a_n = \rho \frac{c \left\| u\right\| }{c\left\| u\right\| } = \rho.
	\]
	This completes the proof. 
\end{proof}

At the end of this section, we select suitable parameters to numerically simulate the dynamics of system \eqref{eq4.12} in \textbf{Cases II} and \textbf{VII}, as follows:
\begin{gather*}
\Omega=(0,1), \quad m_{11}=-1, \quad m_{21}=1, \quad \epsilon=0.01, \quad \kappa=1, \quad \beta(s)=1+9e^{-s}, \\[1mm]
\mu_1(x)=\bar{\mu}_1 + 0.05\sin(2\pi x), \quad 
\mu_2(x)=\bar{\mu}_2 + 0.05\cos(2\pi x), \quad 
 \zeta(x)=1+ \sin (\pi x), \\[1mm]
\quad p_{1}(x,y)= 10 \max \Bigl( 0, 1- 10\min (\left| x-y \right|, 1- \left| x-y \right|)\Bigr), \\[1mm]
p_{2}(x,y)=\frac{20}{3}\max \left( 0, 1- \frac{20}{3}\min (\left| x-y \right|, 1- \left| x-y \right|)\right). 
\end{gather*}
In the following, we consider two parameter sets: case (1) $\bar{\mu}_1=1.0,\ \bar{\mu}_2=0.2$; case (2) $\bar{\mu}_1=0.4,\ \bar{\mu}_2=0.8$.

\renewcommand{\figurename}{Fig.}
\begin{figure}[htbp]
	\centering
	\subcaptionbox{}{\includegraphics[width=0.45\textwidth]{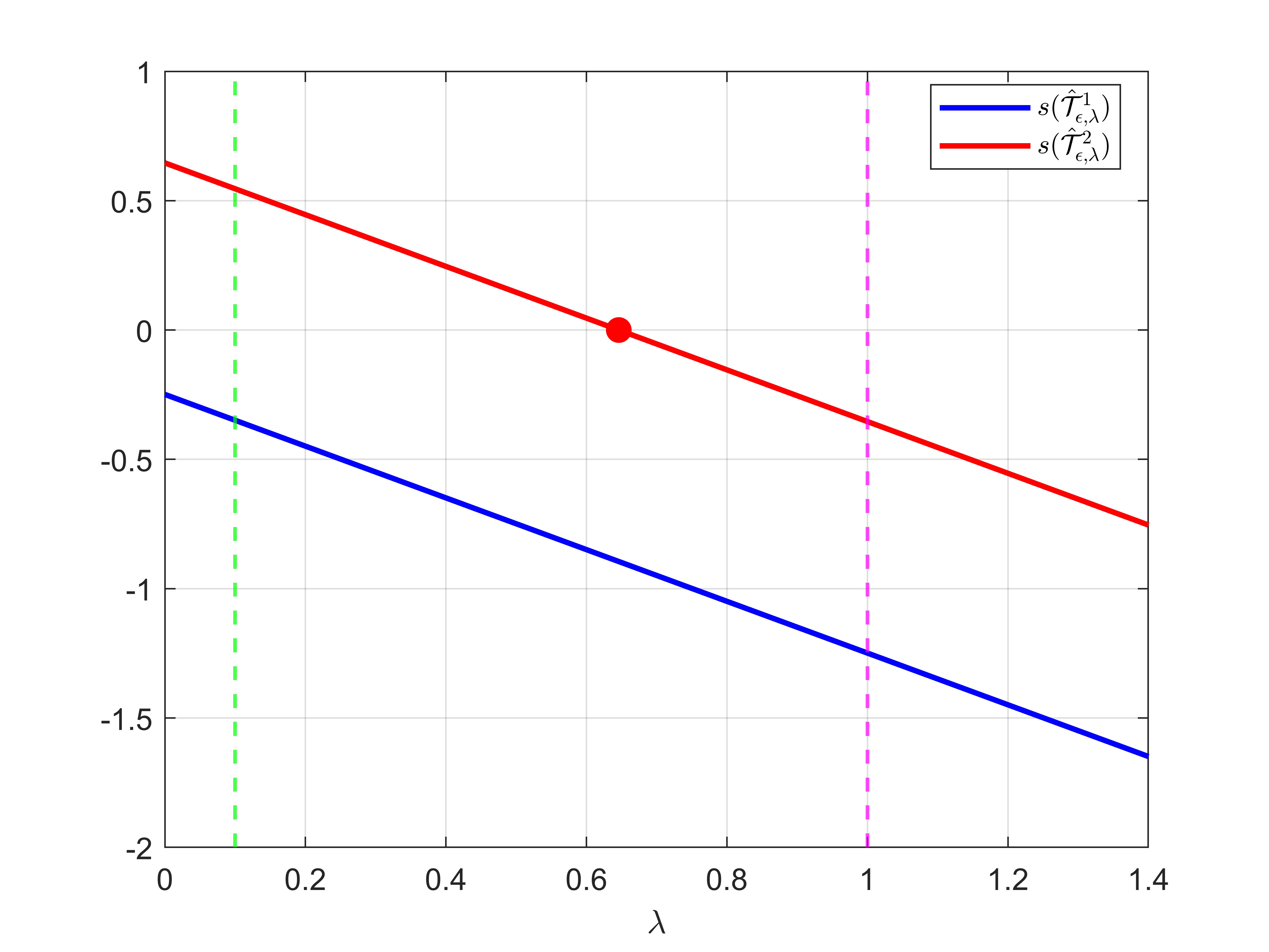}}%
	\hfill
	\subcaptionbox{}{\includegraphics[width=0.45\textwidth]{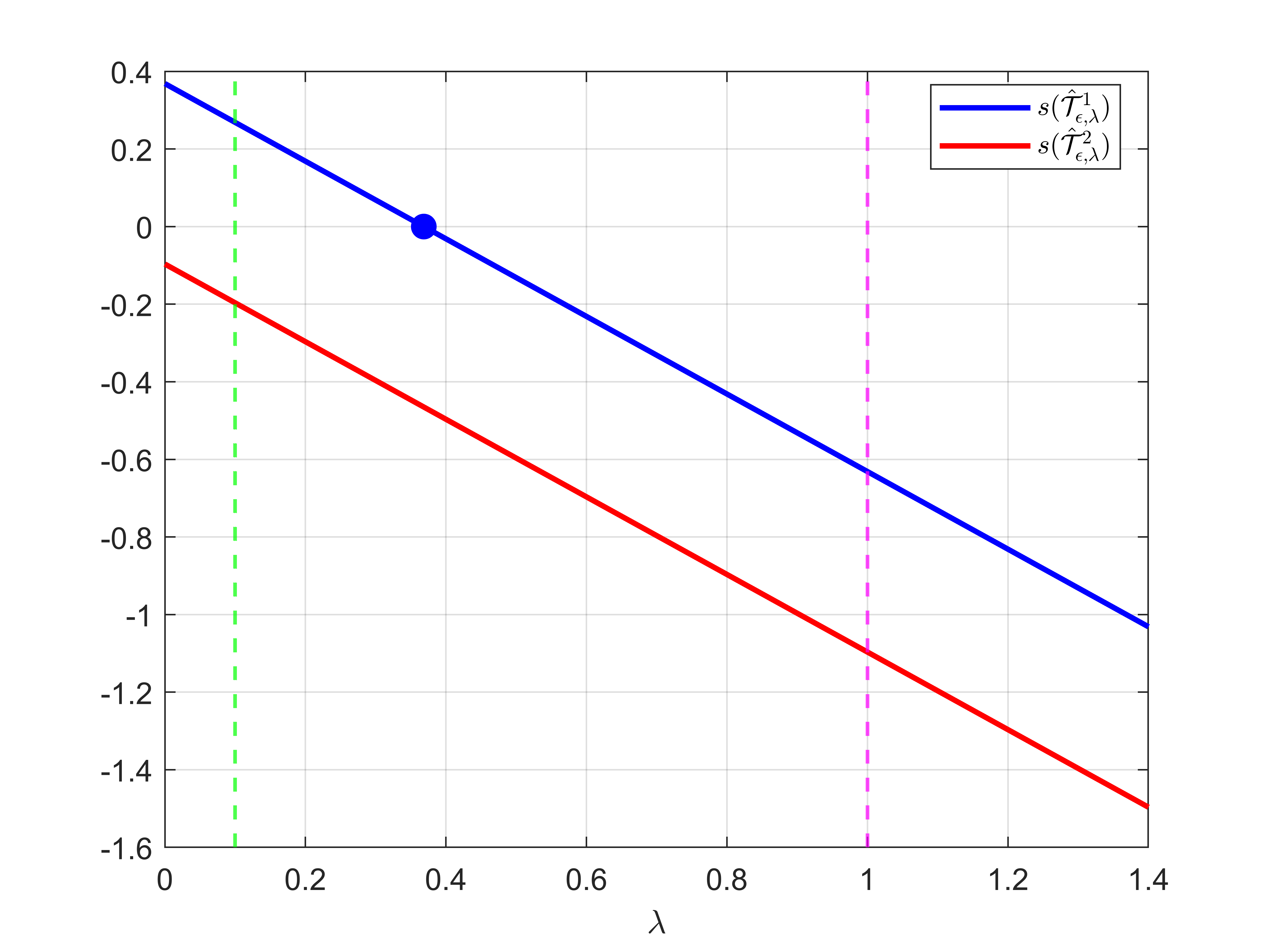}}
	\caption{Plots of \( s(\tilde{\mathcal{T}}_{\epsilon,\lambda}^{i}) \): (a) case (1); (b) case (2)}
	\label{fig1}
\end{figure}

\begin{figure}[htbp]
	\centering
	\subcaptionbox{}{\includegraphics[width=0.45\textwidth]{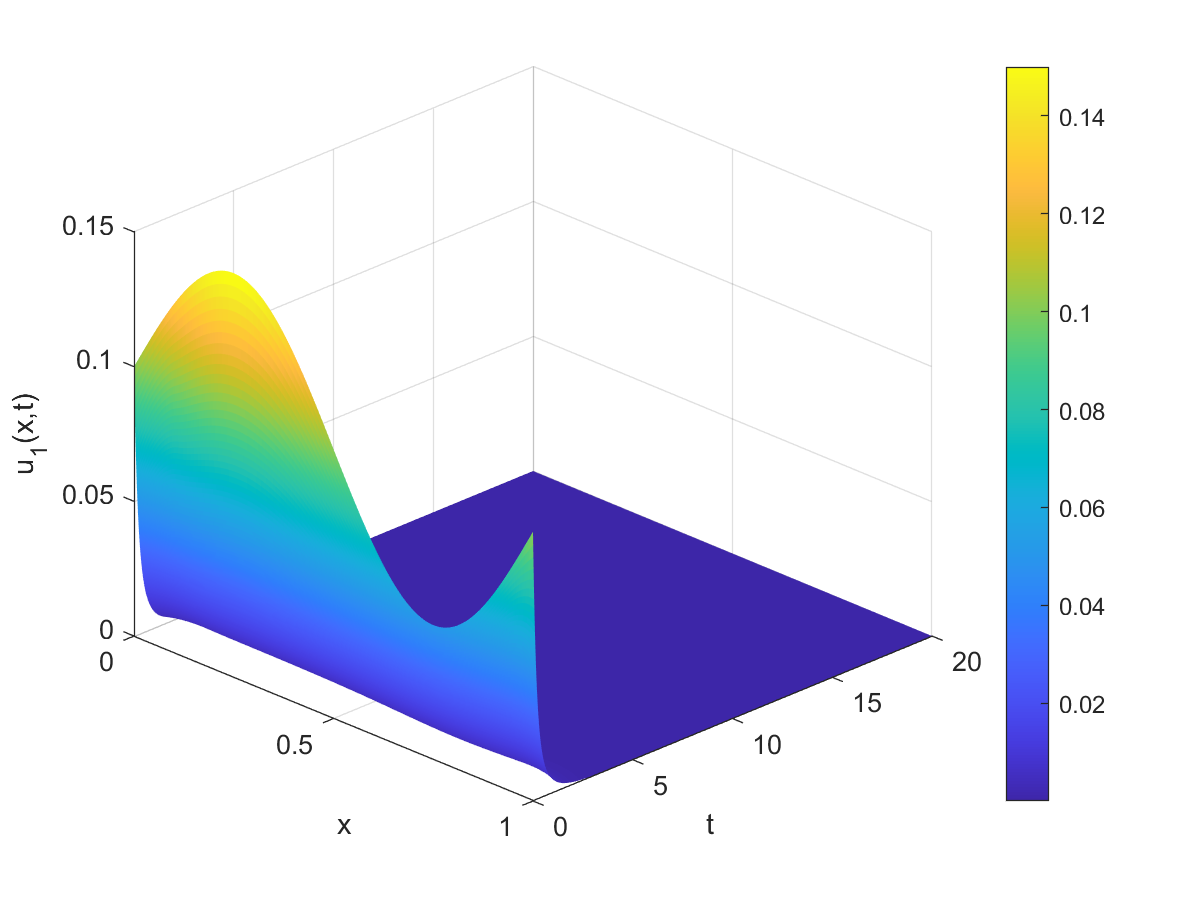}}%
	\hfill
	\subcaptionbox{}{\includegraphics[width=0.45\textwidth]{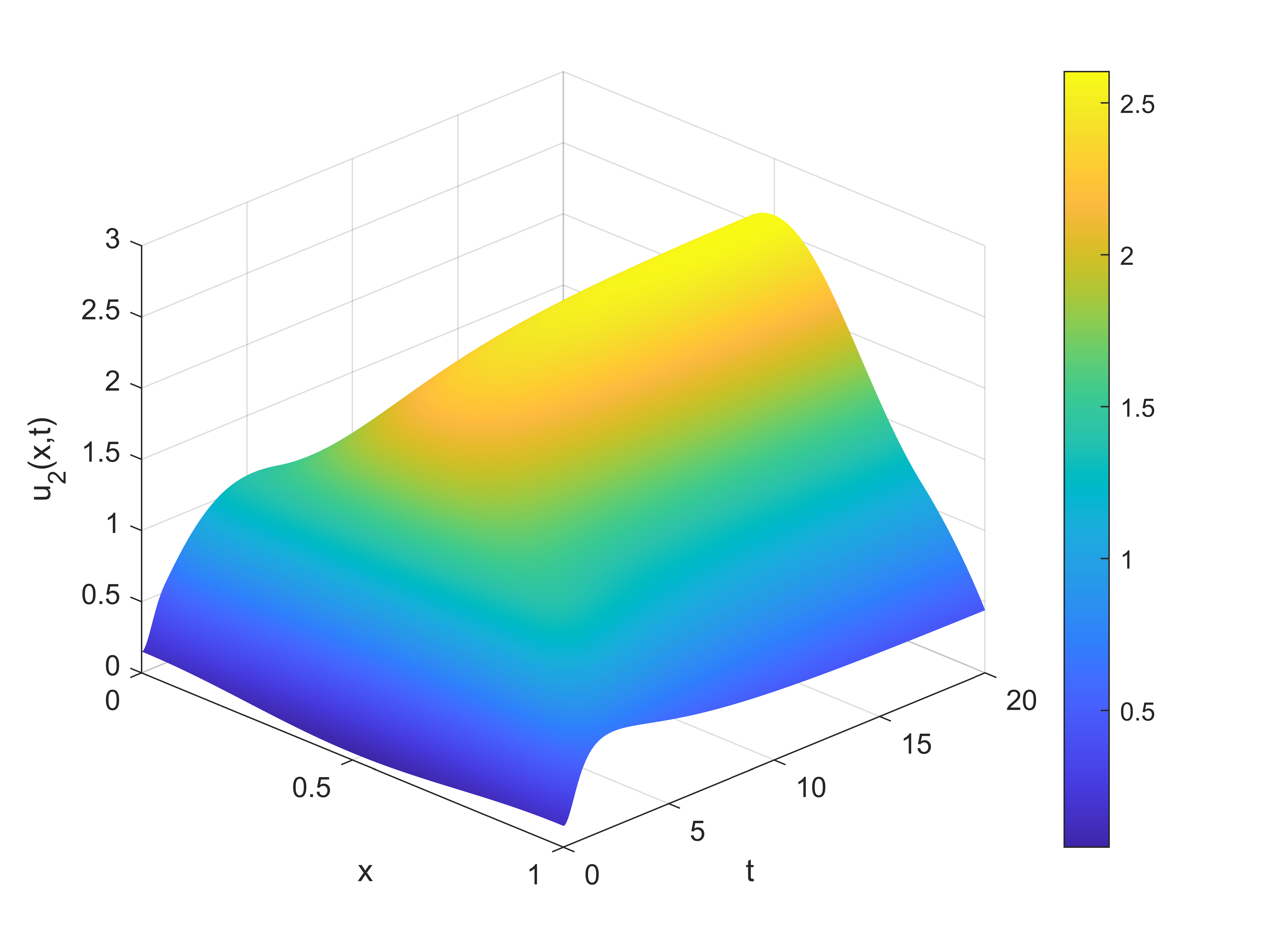}}%
	\\[0.5cm] 
	\subcaptionbox{}{\includegraphics[width=0.45\textwidth]{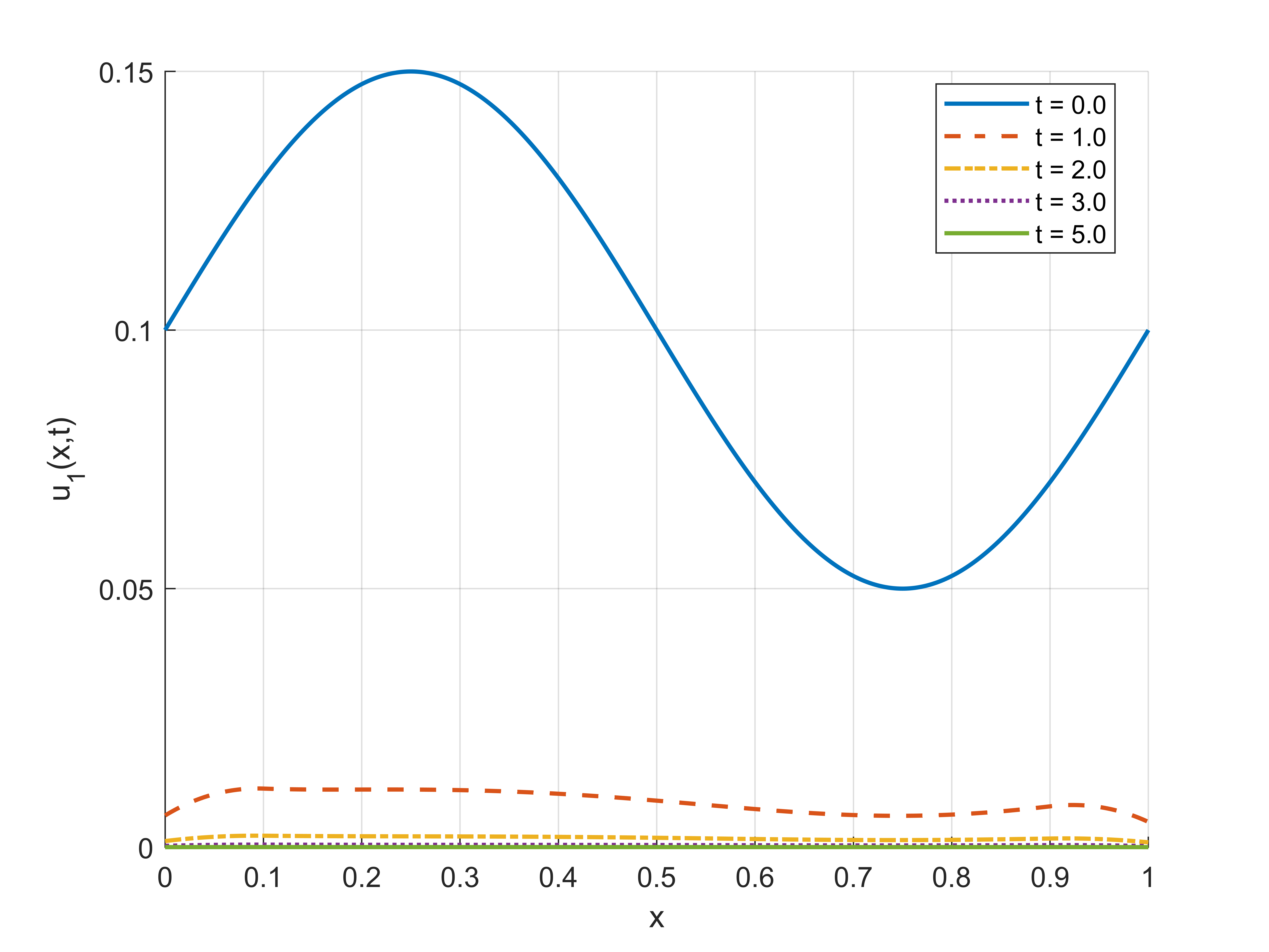}}%
	\hfill
	\subcaptionbox{}{\includegraphics[width=0.45\textwidth]{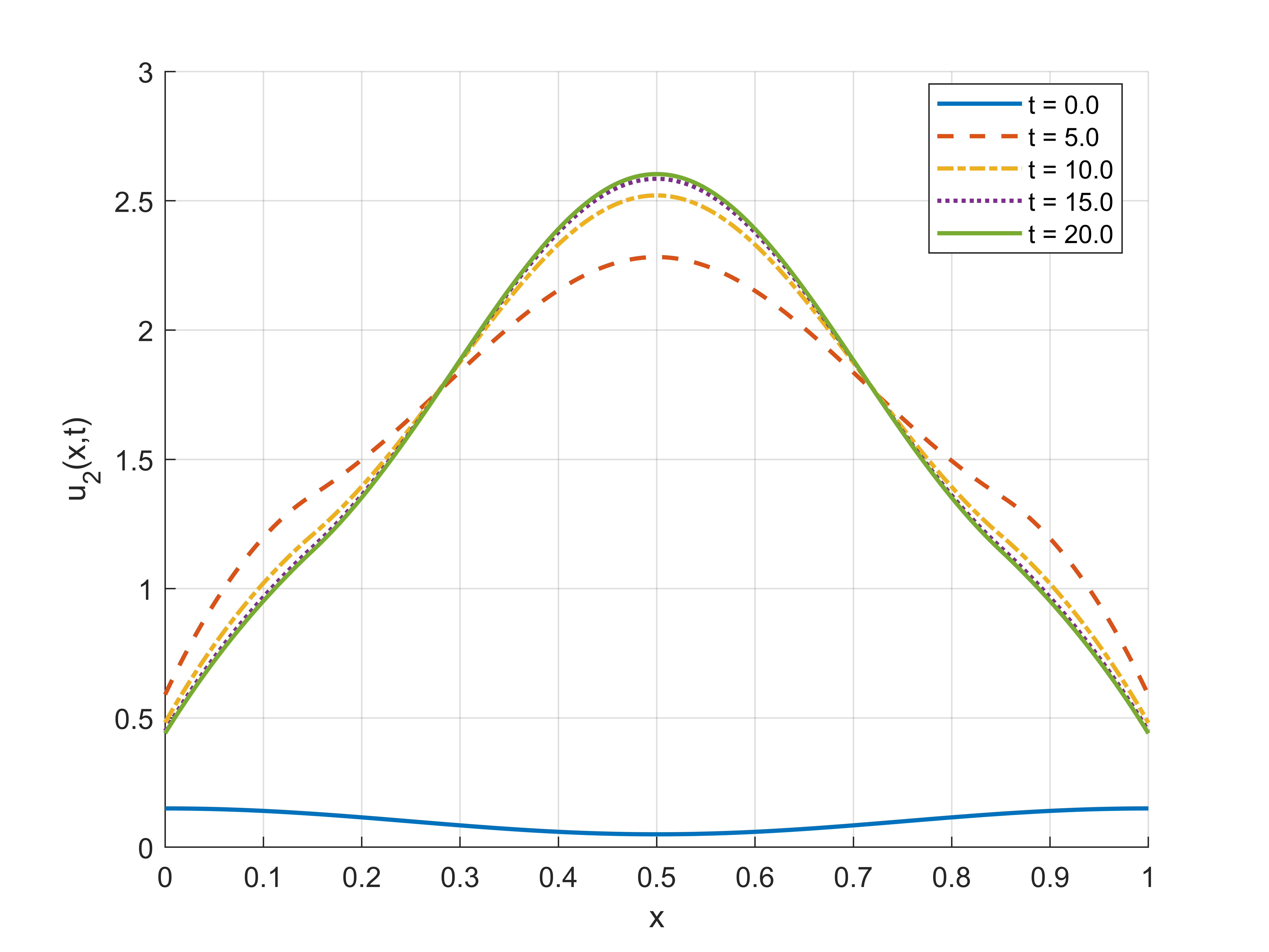}}%
	\caption{Dynamics of system \eqref{eq4.12} in case (1)(\textbf{Case II})}
	\label{fig2}
\end{figure}

From the numerical results shown in Fig. \ref{fig1}, since $\frac{1}{\beta(0)}=0.1$ and $\frac{1}{\beta(+\infty)}=1$, we obtain that in case (1): 
$s(\hat{T}_{\epsilon,0}^{1}) < 0$ and 
$s(\hat{T}_{\epsilon,+\infty}^{2}) < 0 < s(\hat{T}_{\epsilon,0}^{2})$, 
which corresponds to \textbf{Case II}; and in case (2): 
$s(\hat{T}_{\epsilon,+\infty}^{1}) < 0 < s(\hat{T}_{\epsilon,0}^{1})$ and $s(\hat{T}_{\epsilon,0}^{2}) < 0$, which corresponds to \textbf{Case VII}.

We choose the initial values as
\[
u_1(x,0)=0.1+0.05\sin(2\pi x),\quad
u_2(x,0)=0.1+0.05\cos(2\pi x),\quad x\in[0,1].
\]
From Figs. \ref{fig2} and \ref{fig3}, we observe that the numerical simulation results are fully consistent with our theoretical proofs. Specifically, in case (1) (\textbf{Case II}), $u_1$ rapidly approaches zero, while $u_2$ converges to a positive steady state. In case (2) (\textbf{Case VII}), both $u_1$ and $u_2$ converge to positive steady states.

\begin{figure}[htbp]
	\centering
	\subcaptionbox{}{\includegraphics[width=0.45\textwidth]{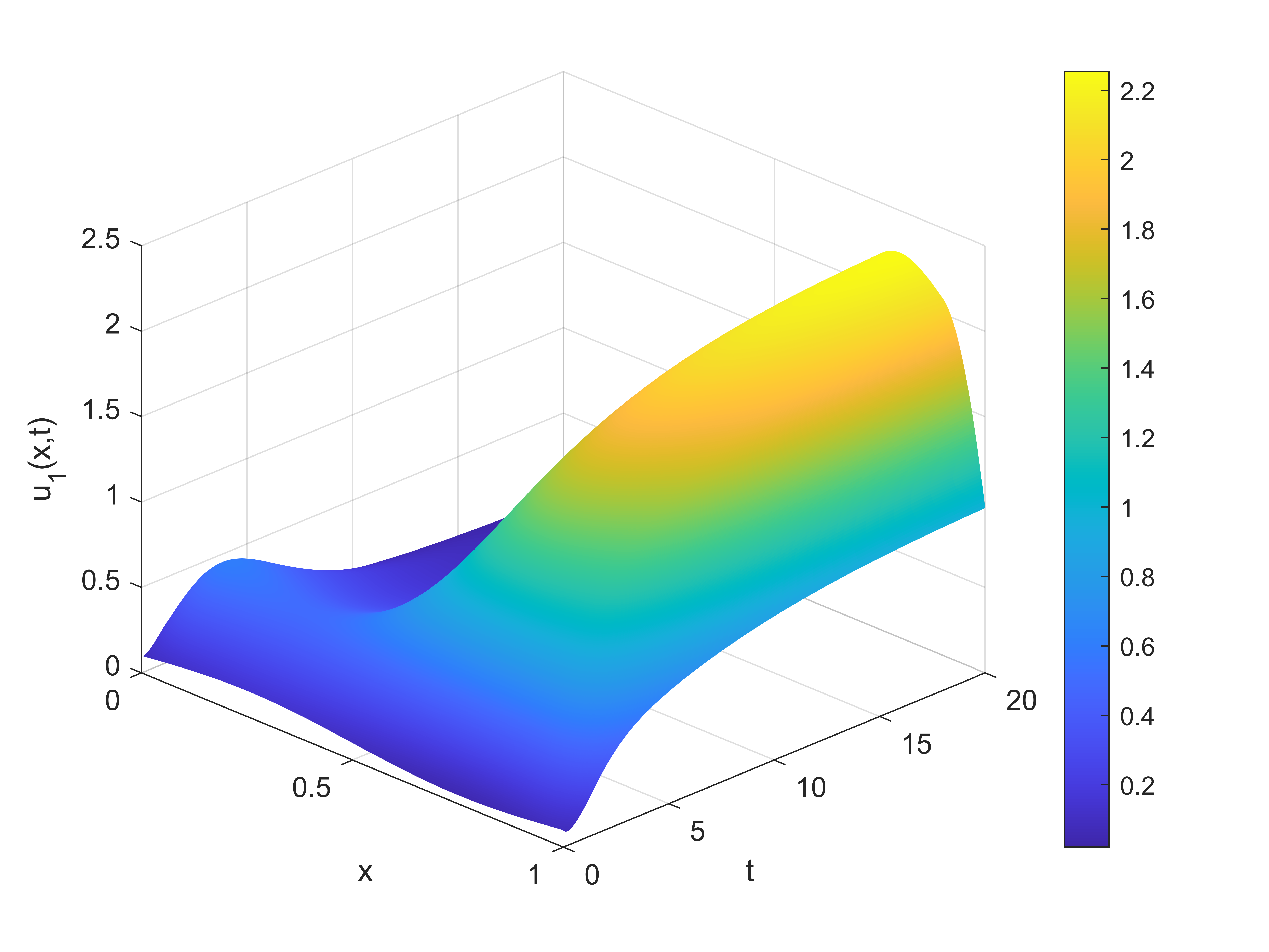}}%
	\hfill
	\subcaptionbox{}{\includegraphics[width=0.45\textwidth]{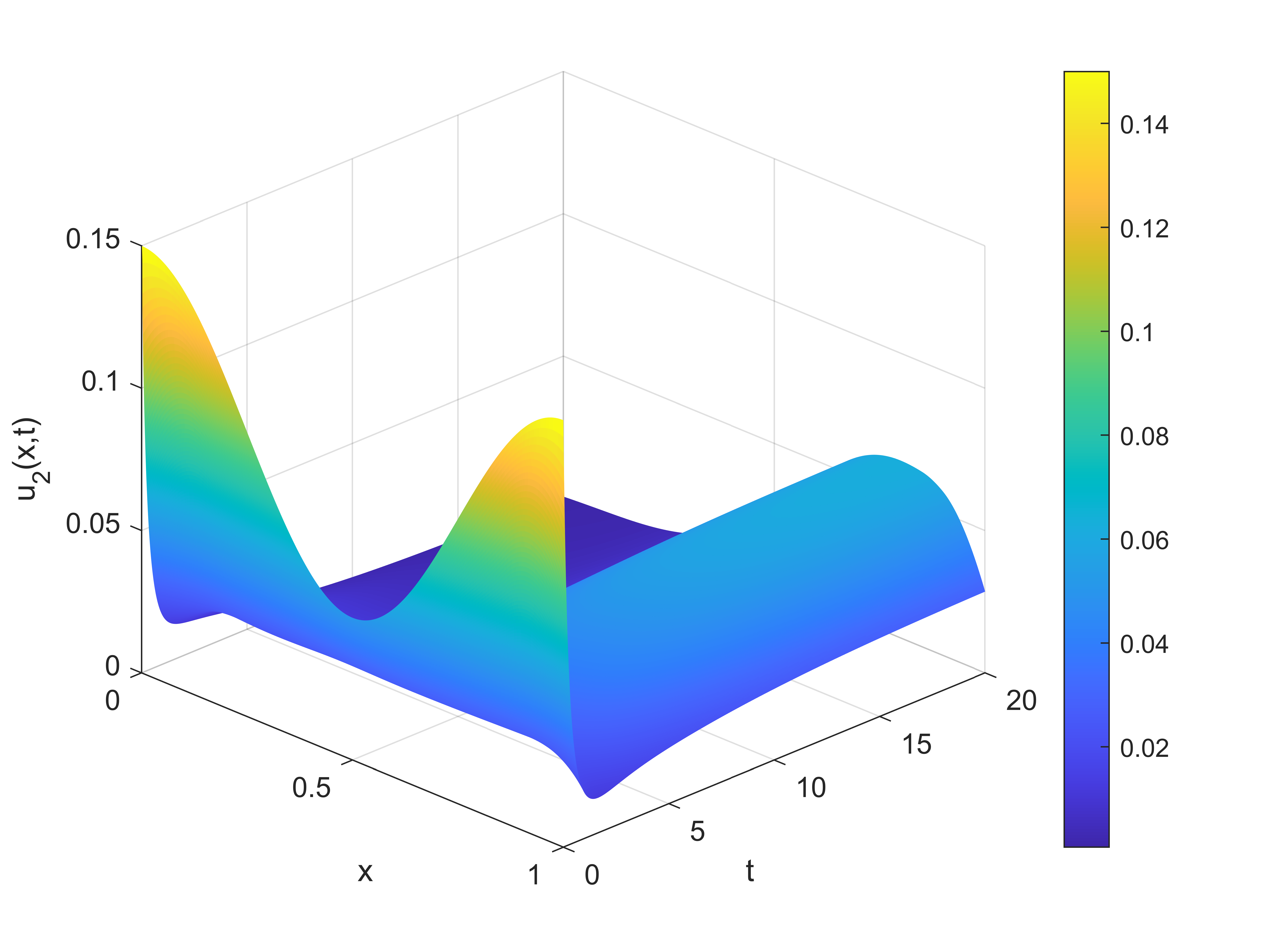}}%
	\\[0.5cm] 
	\subcaptionbox{}{\includegraphics[width=0.45\textwidth]{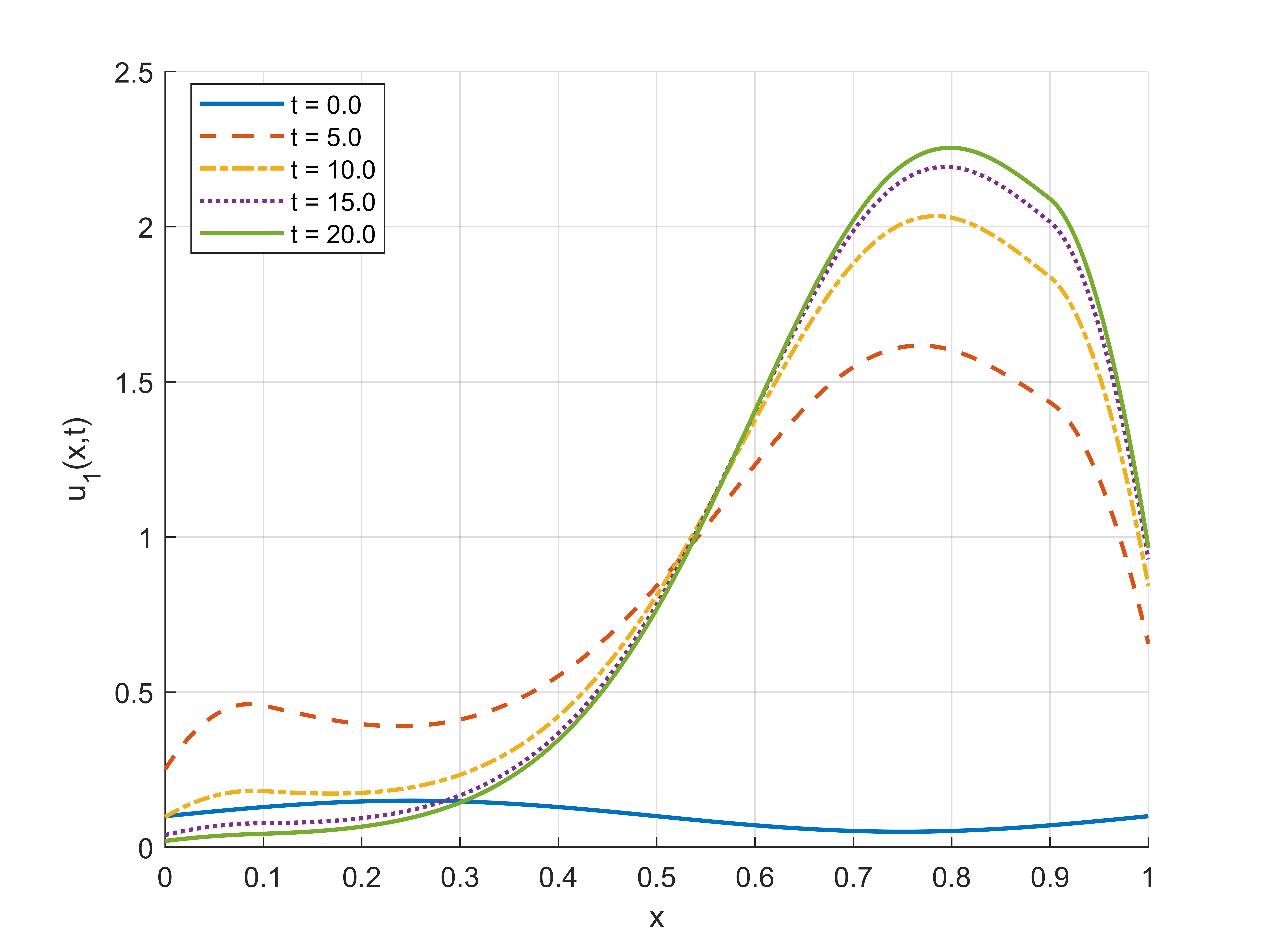}}%
	\hfill
	\subcaptionbox{}{\includegraphics[width=0.45\textwidth]{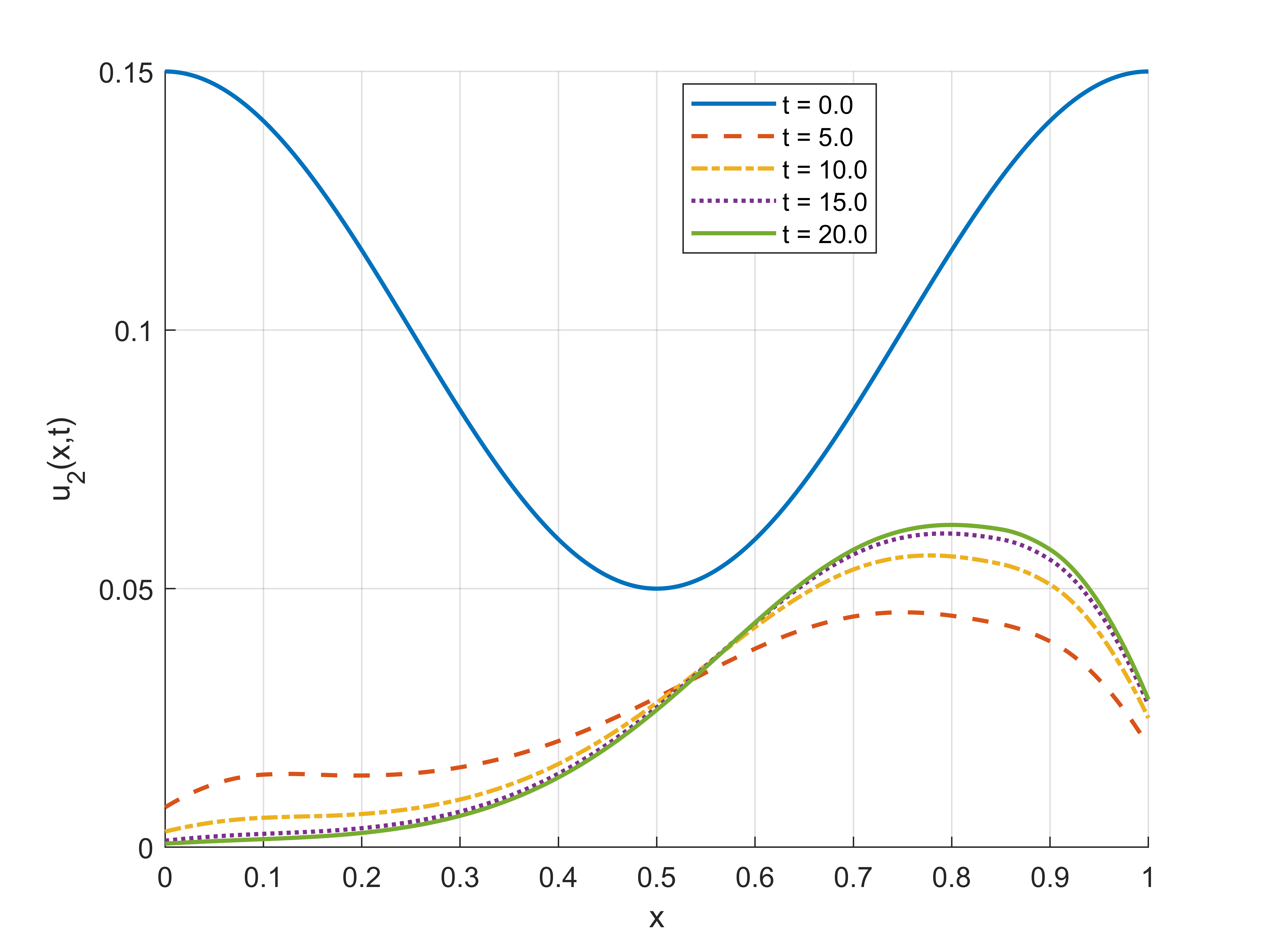}}%
	\caption{Dynamics of system \eqref{eq4.12} in case (2)(\textbf{Case VII})}
	\label{fig3}
\end{figure}

\section*{Availability of data and material}
Not applicable.

\section*{Competing interests}
The authors declare that they have no competing interests.

\section*{Acknowledgments}
This work  was supported  by the National Natural Science Foundation of China (Nos. 12471176 and 12071491) and Guangdong Basic and Applied Basic Research Foundation (No. 2025A1515012221).  X. Lin was additionally supported  by the China Postdoctoral Science Foundation (No. 2025M783159).


\end{document}